\pdfoutput=1
\documentclass[a4paper,11pt,reqno]{amsart}

\usepackage[T1]{fontenc}
\usepackage[USenglish]{babel}
\usepackage[utf8]{inputenc}
\usepackage{amsfonts} 
\usepackage{amsmath} 
\usepackage[alphabetic]{amsrefs} 
\usepackage{amssymb} 
\usepackage{comment}
\usepackage{newtxtext}
\usepackage{amsthm} 
\usepackage{booktabs}
\usepackage{xcolor} 
\usepackage{fix-cm}
\usepackage{bm}
\usepackage{derivative}
\usepackage{fancyhdr}
\usepackage{array}
\usepackage{braket}
\usepackage{scalefnt}
\usepackage{lipsum}

\usepackage[labelfont={bf},hypcap=false]{caption}
\usepackage{mathtools}
\usepackage{thmtools}
\usepackage{graphicx}
\usepackage[normalem]{ulem}
\usepackage{tabularx}

\usepackage[colorlinks]{hyperref}
\hypersetup{
  linkcolor=[RGB]{192,0,0},
  citecolor=[RGB]{0, 120, 0},
  urlcolor=[rgb]{0.6, 0.2, 0.2}
}

\usepackage{paralist}
\usepackage{tensor}

\usepackage{tikz}
\usepackage{tikz-cd}
\usepackage{tikz-3dplot}
\usetikzlibrary{cd}
\usetikzlibrary{arrows}
\usetikzlibrary{matrix}
\usepackage{nicematrix} 
\usetikzlibrary{positioning}

\tikzcdset{arrow style=math font}

\usepackage[subscriptcorrection]{newtxmath} 

\DeclareMathAlphabet{\mathbb}{U}{msb}{m}{n}
\DeclareMathAlphabet{\mathfrak}{U}{euf}{m}{n} 
\DeclareMathAlphabet{\mathbb}{U}{msb}{m}{n}

\usepackage{xcolor}
\definecolor{red}{rgb}{1,0,0}
\definecolor{darkred}{RGB}{192,0,0}

\newcommand{\binomial}[2]{\left( \begin{array}{c}\hspace{-5pt} #1 \\ \hspace{-5pt}#2 \end{array}\hspace{-5pt}\right)}

\newcommand{\bfa}{\mathbf{a}}

\newcommand{\bfm}{\mathbf{m}}

\newcommand{\bfx}{\mathbf{x}}
\newcommand{\bfy}{\mathbf{y}}

\newcommand{\calA}{\mathcal{A}}

\newcommand{\calH}{\mathcal{H}}

\newcommand{\calP}{\mathcal{P}}

\newcommand{\bbC}{\mathbb{C}}

\newcommand{\bbK}{\mathbb{K}}

\newcommand{\bbN}{\mathbb{N}}

\newcommand{\bbR}{\mathbb{R}}

\newcommand{\mfS}{\mathfrak{S}}

\newcommand{\rme}{\mathrm{e}}

\newcommand{\rmi}{\mathrm{i}}

\newcommand{\rmp}{\mathrm{p}}

\newcommand{\bfalpha}{\boldsymbol{\alpha}}

\newcommand{\pder}{\mathop{}\!\mathrm{\partial}}

\newcommand{\Lap}{\mathop{}\!\Delta}

\DeclareMathOperator{\brk}{brk}
\DeclareMathOperator{\Cat}{Cat}

\DeclareMathOperator{\p}{p}

\DeclareMathOperator{\rk}{rk}

\DeclarePairedDelimiter{\abs}{\lvert}{\rvert}

\newcommand{\Oa}{\mathrm{O}}

\usepackage[left=1in, right=1in, top=1in, bottom=1in, includefoot, headheight=13.6pt]{geometry}

\usepackage{enumitem}

\usepackage{adjustbox}

\usepackage{titlesec}
\titleformat{\section}{
 \vspace{2pt}\scshape\fontfamily{ptm}\raggedright\large}{}{0em}{\hspace{-0.4pt}\large \thesection.\hspace{0.5em}}[\color{black}\titlerule \vspace{-1pt}]
 
\titleformat{name=\section,numberless}
{\vspace{2pt}\scshape\fontfamily{ptm}\raggedright\large}{}{0em}
  {\hspace{-0.4pt}\large}[\color{black}\titlerule \vspace{-1pt}]

\allowdisplaybreaks

\numberwithin{equation}{section}
\theoremstyle{definition}
\newtheorem{defn}[equation]{Definition}
\theoremstyle{plain}
\newtheorem{teo}[defn]{Theorem}
\newtheorem{prop}[defn]{Proposition}

\newtheorem{lem}[defn]{Lemma}

\newtheorem{cor}[defn]{Corollary}
\theoremstyle{remark}

\newtheorem{exam}[defn]{Example}
\allowdisplaybreaks

\renewcommand{\sectionmark}[1]{\markboth{\normalfont \scshape\fontfamily{ptm}\selectfont Cosimo Flavi}{\normalfont \scshape\fontfamily{ptm}\selectfont Upper bounds for the rank of powers of quadrics}}
\title{\large Upper bounds for the rank of powers of quadrics}
\author{Cosimo Flavi}
\address{{\normalfont (Cosimo Flavi)},
\normalfont \scshape\fontfamily{ptm}\selectfont
	Università degli Studi di Firenze, 
	Dipartimento di Matematica e Informatica "Ulisse Dini", \normalfont{Viale Giovanni Battista Morgagni 67/a, 50134 Florence, Italy.}}
\email{cosimo.flavi@unifi.it}



\ExplSyntaxOn
\prg_new_conditional:Npnn \int_if_zero:n #1 { p , T , F , TF }
  {
    \if_int_compare:w \int_eval:w #1 = \c_zero_int
      \prg_return_true:
    \else:
      \prg_return_false:
    \fi:
  }
\ExplSyntaxOff

\begin{document}
\renewcommand{\abstractname}{\normalfont \scshape\fontfamily{ptm}\selectfont{Abstract}}
\begin{abstract}
We establish an upper bound for the rank of every power of an arbitrary quadratic form. Specifically, for any $s\in\bbN$, we prove that the $s$-th power of a quadratic form of rank $n$ grows as $n^{s}$. Furthermore, we demonstrate that its rank is subgeneric for all $n>(2s-1)^2$.
\end{abstract}
\maketitle
\thispagestyle{empty}
\section{Introduction}
\markboth{\normalfont \scshape\fontfamily{ptm}\selectfont Cosimo Flavi}{\normalfont \scshape\fontfamily{ptm}\selectfont Upper bounds for the rank of powers of quadrics}
Quadratic forms are common objects in many branches of mathematics.
The general theory of quadratic forms over the rationals and rational integers dates back to the late 19\textsuperscript{th} century and it was first introduced by H.~Minkowski (see \cite{Min84}). 
A simple and well-known fact is that any quadratic form is equivalent to the form
\[
q_n=x_1^2+\cdots+x_n^2
\]
for some $n\in\bbN$ (see, e.g.~\cite{Ser73}*{chapter 4}). Then, for any $s\in\bbN$, $q_n^s$ is a $2s$-degree form which is invariant under the action of the orthogonal group $\Oa_n(\bbC)$. 

We recall that, for any homogeneous polynomial $f\in\bbC[x_1,\dots,x_n]$ of degree $d$, a \textit{Waring decomposition}, or simply \textit{decomposition}, of $f$ of \textit{size} $r$ is
an expression of $f$ as a sum of $d$-th powers of $r$ different linear forms $l_1,\dots,l_r\in\bbC[x_1,\dots,x_n]$, which can be viewed as points in $\bbC^n$, that is,
\[
f=\sum_{j=1}^rl_j^d.
\] 
The \textit{Waring rank} (or simply \textit{rank}) of $f$, denoted by $\rk f$, is the smallest natural number $r$ such that there exists a decomposition of $f$ of size $r$. Currently, obtaining the rank of a given polynomial still remains a challenging problem and there is no general efficient method to solve it, except for the special case of two variables, first analyzed  in \cite{Syl51}and completely solved in \cite{CS11}.  
Nevertheless, many partial results and methods have been developed over the years. For a more detailed overview of Waring decompositions, see for example \cites{BCC+18,BGI11,CGO14,Lan12,LO13}, where many applications are also presented, such as telecommunications in electrical engineering (see e.g.~\cite{Che99} and \cite{DC07}) or cumulant tensors in statistics (see e.g.~\cite{McC87}).
 
The problem of determining the rank of a generic form was solved by J.~Alexander and A.~Hirschowitz in \cite{AH95}, who proved that, except for a few cases of degree $2$, $3$, and $4$,
there exists a Zariski open set $\Omega\subseteq\bbC[x_1,\dots,x_n]_d$ such that, for every $f\in\Omega$,
\[
\rk f=\biggl\lceil\frac{1}{n}\binom{d+n-1}{d}\biggr\rceil\sim n^{d-1}.
\]
More recently, M.~C.~Brambilla and G.~Ottaviani provided a shorter version of the proof in \cite{BO08}, which we refer to for the details.

In the present paper, we focus on the decompositions and Waring rank of powers of quadratic forms.
In general, determining the rank of $q_n^s$ is a hard problem.  While in the case of binary forms the problem is quite simple and completely solved (\cite{Rez92}*{Theorem 9.5}), we cannot say the same for the general case in more variables, about which there is not much information in the literature. 
However, we focus on the asymptotic behavior of the rank of $q_n^s$, which grows as $n^s$, with $n\to+\infty$.

A particularly remarkable fact about quadratic forms is their invariance under the action of orthogonal groups. In particular, this implies that any decomposition of the form $q_n^s$ has an $n$-dimensional orbit under the action of the orthogonal group $\Oa_n(\bbC)$.
The most complete study of this subject so far has been given by B.~Reznick, who in \cite{Rez92} provides a precise survey of both classical and more original results over the field $\bbR$. In particular, B.~Reznick concentrates in his notes on real Waring decompositions, which he calls \textit{representations}. 
The author also lists several applications of decompositions of powers of quadratic forms in \cite{Rez92}*{Section 8}, including their use in number theory for studying the Waring problem and in functional analysis.
The quadratic form $q_n$, in terms of differential operators, is exactly the well-known Laplace operator
\[
\Lap=\pdv[2]{}{x_1}+\cdots+\pdv[2]{}{x_n},
\]
which also plays an important role also in mathematical analysis.
From an algebraic point of view, the equality
\[
S^d\bbC^n=\bigoplus_{j=0}^{\left\lfloor\frac{d}{2}\right\rfloor}q_n^j\calH_{n,d-2j},
\]
where, $\calH_{n,d-2j}$ denotes the space of homogeneous harmonic polynomials of degree $d-2j$ in $n$ variables, for every $j\in\bbN$, provides a decomposition of great relevance. This elegant formula, analyzed in detail by R.~Goodman and N.~R.~Wallach in \cite{GW98}*{Corollary 5.2.5}, gains much importance by considering its invariance under the action of the complex orthogonal group $\Oa_n(\bbC)$. 

In \autoref{sec_powers_of_quadrics}, we present some of the most important decompositions analyzed by B.~Reznick, focusing on closed formulas that provide a family of decompositions depending on the number of variables. In general, these formulas do not correspond to minimal decompositions, but can provide insights into how the rank grows as the number of variables tends to infinity. We also examine the cases in which we can establish that the rank of $q_n^s$ is subgeneric. A first example is given by the formula
\begin{equation}
\label{rel_decom_Rezn_q_n^2}
6q_n^2=\sum_{i_1<i_2}(x_{i_1}\pm x_{i_2})^4+2(4-n)\sum_{i_1}x_{i_1}^4,
\end{equation}
which is a decomposition of size $n^2$, for $n\neq 4$, presented by B.~Reznick in \cite{Rez92}*{formula (8.33)}. As we will see in \autoref{teo_minimal_decomposition_q_n^2}, this formula does not provide minimal decompositions for every $n\in\bbN$. However, it exhibits an elegant pattern and a structure that is invariant under the action of the permutation group $\mfS_n$. Furthermore, it yields subgeneric decompositions for $n>17$. 
This formula has already been taken into consideration by J.~Buczyński, K.~Han, M.~Mella, and Z.~Teitler in \cite{BHMT18}*{section 4.5}, where they provide a similar decomposition for the exponent $3$. The latter follows the same pattern of points (cf.~\autoref{prop_BHMT18_q_n^3}). It is given by the formula 
\begin{equation}
\label{rel_decomp_BHMT18_q_n^3}
60q_n^3=\sum_{i_1<i_2<i_3}(x_{i_1}\pm x_{i_2}\pm x_{i_3})^6+2(5-n)\sum_{i_1<i_2}(x_{i_1}\pm x_{i_2})^6+2(n^2-9n+38)\sum_{i_1}x_{i_1}^6.
\end{equation}
We generalize this pattern in \autoref{sec_further examples}, providing other formulas for higher powers. Unfortunately, describing the pattern is not as easy as one might expect, since linear forms involving only $1$ and $-1$ as coefficients are not sufficient to cover the coefficients of all monomials of degree $2s$. 
To obtain a general formula for any $s$-th power, it is necessary to consider other types of linear forms. In order to do this, we need to recall the classical notion of $k$-partition function for every $k\in\bbN$. This function, denoted by $\p_k\colon\bbN\to\bbN$, associates every natural number $s$ with the number of partitions of $s$ into exactly $k$ parts. Based on this, we provide an upper bound for the rank of $q_n^s$, which is summarized in the following theorem, which is proved in \autoref{sec_Asymptotic_growing}.
\begin{teo}
\label{cor_upper_bound_quadrics}
For every $n,s\in\bbN$
\begin{equation}
\label{rel_formula_upper_bound_corollary}
\rk(q_n^s)\leq 2^{s-1}\binom{n}{s}+2^{s-2}\binom{n}{s-1}+\sum_{k=1}^{s-2}2^{k-1}k!\p_k(s)\binom{n}{k}.
\end{equation}
In particular, for any $s\in\bbN$, the rank of $q_n^s$ grows at most as $n^s$. 
\end{teo}
Given any $s\in\bbN$, the generic rank of forms of degree $2s$ grows as $n^{2s-1}$ when $n\to +\infty$. Hence, by the theorem of J.~Alexander and A.~Hirschowitz (\cite{AH95}) mentioned above, the last theorem guarantees that for any $s\in\bbN$, the rank can be generic or supergeneric only for a finite number of possible values of $n$.  

For any form $f$, the \textit{border rank} of $f$ , denoted by $\brk f$, is the smallest natural number $r\in\bbN$ such that $f$ is in the Zariski closure of a set of homogeneous polynomials having rank equal to or less than $r$.

Thanks to a result due to B.~Reznick in \cite{Rez92} (see \autoref{prop_lower_bound_catalecticant_powers_quadrics}), we know that
\begin{equation}
\label{rel_lower_bound_catalecticant_Reznick}
\rk(q_n^s)\geq\brk(q_n^s)\geq\binom{s+n-1}{s}.
\end{equation}
Putting together \autoref{cor_upper_bound_quadrics} and inequality \eqref{rel_lower_bound_catalecticant_Reznick}, we have the following result, which is also proved in \autoref{sec_Asymptotic_growing}.
\begin{cor}
\label{cor_log_limit}
For every $s\in\bbN$,
\[
\lim_{n\to +\infty}\log_n\bigl(\rk(q_n^s)\bigr)=\lim_{n\to +\infty}\log_n\bigl(\brk(q_n^s)\bigr)=s.
\]
\end{cor}
This fact agrees perfectly with the value of the border rank determined in \cite{Fla22}*{Theorem 4.5} for ternary non-degenerate quadratic forms, which turns out to be equal to the rank of the middle catalecticant matrix, that is,
\[
\brk(q_3^s)=\binom{s+2}{2}.
\]

In \autoref{sec:subgeneric_rank}, we provide an estimate for the maximum value $n_s\in\bbN$ for which the rank of $q_n^s$ is subgeneric for every $n>n_s$. 
We state this in the following theorem.
\begin{teo}
\label{teo_subgeneric_rank}
For every $s\in\bbN$, whenever
$n>(2s-1)^2$,
\[
\rk(q_n^s)<\frac{1}{n}\binom{2s+n-1}{2s}.
\]
In particular, the rank of $q_n^s$ is subgeneric, whenever
$n>(2s-1)^2$.
\end{teo}
This result partially solves the problem posed by J.~Buczyński, K.~Han, M.~Mella, and Z.~Teitler in \cite{BHMT18}*{section 4.5}, who asked for the values of $n$ and $s$ for which the rank of $q_n^s$ is subgeneric. Moreover, the authors provide an analysis 
on the maximal rank locus $W_{\mathrm{max},d}$, that is, the set of homogeneous polynomials of degree $d$ having maximum rank. In particular, they prove the following theorem.
\begin{teo}[\cite{BHMT18}*{Theorem 4.1}]
\label{teo_BHTM18_component_maximal_rank_loci}
Let $V$ be a finite dimensional vector space with $\dim V=n$ such that $n\geq 3$. If $W$ is an irreducible component of the rank locus $W_{\mathrm{max},d}$, then \[
\dim(W)\geq\binom{n+1}{2}-1.
\]
Moreover, if
\[
\dim(W)=\binom{n+1}{2}-1,
\]
then $d$ is even, and $W$ is the set of all the $\bigl(d/2\bigr)$-th powers of quadrics.
\end{teo}
 With the result of \autoref{teo_subgeneric_rank}, we improve this last statement.
By \autoref{teo_subgeneric_rank}, we can state, in particular, that the second condition of \autoref{teo_BHTM18_component_maximal_rank_loci} does not hold for $n>(2s-1)^2$.
\begin{cor}
\label{cor:improvement_result}
For every $s\in\bbN$ and for every $n\geq 3$, if $n>(2s-1)^2$ and $W$ is an irreducible component of $W_{\mathrm{max},2s}$, then
\[
\dim(W)\geq\binom{n+1}{2}.
\]
\end{cor}
\section{Powers of quadrics and classical decompositions}
\label{sec_powers_of_quadrics}
Thorough the paper, we will use the notation
\[
\sum(a_1x_1\pm\cdots\pm a_nx_n)^m
\]
to denote a summation of all the $2^{n-1}$ possible choices of the coefficients $+1$ and $-1$, multiplying $a_2,\dots,a_n$. In other words, 
\[
\sum(a_1x_1\pm\cdots\pm a_nx_n)^m=\sum_{\substack{j_i\in\{0,1\}\\i=2,\dots,n}}\bigl(a_1x_1+(-1)^{j_2}a_2x_2+\cdots+(-1)^{j_n}a_nx_n\bigr)^m,
\]
where $a_1,\dots,a_n\in\bbC$ and $m\in\bbN$. We will also use the notation 
\[
\sum_{i_1<\cdots<i_k}(a_1x_{i_1}\pm\cdots \pm a_k x_{i_k})^m
\]
to denote a summation over any choice of $k$ variables $x_{i_1},\dots,x_{i_k}$ such that
$1\leq i_1<\cdots<i_k\leq n$, 
where $a_1,\dots,a_n\in\bbC$ and $m\in\bbN$. In the case where $k>n$, we assume the terms $x_{i_{n+1}},\dots,x_{i_k}$ to be $0$, so that in this case, we have
\[
\sum_{i_1<\cdots<i_k}(a_1x_{i_1}\pm\cdots \pm a_k x_{i_k})^m=\sum_{i_1<\cdots<i_n}(a_1x_{i_1}\pm\cdots \pm a_k x_{i_n})^m.
\]
Determining decompositions of the polynomial $q_n^s$, for any $n,s\in\bbN$, is a classical problem. 
Many decompositions appear in both old and more recent literature. B.~Reznick provides an overview of classical results, along with new decompositions, in \cite{Rez92}*{chapters 8-9}. Here are some examples.

Using the notation
$\bfx^{\bfalpha}=x_1^{\alpha_1}\cdots x_n^{\alpha_n}$,
for every multi-index $\bfalpha=(\alpha_1,\dots,\alpha_n)\in\bbN^n$, the \textit{catalecticant map} of a homogeneous polynomial $f$ of degree $d$ is obtained by extending the map defined on monomials as
\[
\begin{tikzcd}[row sep=0pt,column sep=1pc]
\Cat_f\colon \bbC[y_1\dots,y_n]\arrow{r} & \bbC[x_1\dots,x_n]\hphantom{.} \\
  {\hphantom{\Cat_f\colon{}}} \bfy^{\bfalpha} \arrow[mapsto]{r} &\displaystyle \frac{\pder^{\abs{\bfalpha}}f}{\pder\bfx^{\bfalpha}}.
\end{tikzcd}
\]

Thanks to the catalecticant map, it is possible to establish a general inequality that provides a well-known lower bound for the Waring rank, related to catalecticant matrices and apolarity (see \cite{IK99} for an overview on this subject).
We denote by $\Cat_f^k$ the $k$-th component of catalecticant map of $f$, which is the restriction 
\[
\Cat_f^k\colon\bbC[y_1,\dots,y_n]_k\to\bbC[x_1,\dots,x_n]_{d-k}
\]
of the catalecticant map of $f$ to the component of degree $k$ of $\bbC[x_1,\dots,x_n]$. Then we have
\begin{equation}
\label{rel_lower_bound}
\rk f\geq\brk f\geq\rk\bigl(\Cat_f^k\bigr)
\end{equation}
for every $k\in\bbN$. This inequality is classically attributed to J.~J.~Sylvester for binary forms (see \cite{Syl51}) and was then generalized for an arbitrary number of variables (see \cite{IK99}*{page 11}). It appears several times in the literature (see for example \cite{Lan12}*{Proposition 3.5.1.1}).
In particular, B.~Reznick shows that all the catalecticant matrices of $q_n^s$ are full rank, using \cite{Rez92}*{Theorem 8.15} and referring to \cite{Rez92}*{Theorems 3.7 and 3.16}. A further proof of this fact is provided by F.~Gesmundo and J.~M.~Landsberg in \cite{GL19}*{Theorem 2.2}. Moreover, it follows directly from the structure of the kernel of the catalecticant map of $q_n^s$, also known as the \textit{apolar ideal} of $q_n^s$ and denoted by $(q_n^s)^{\perp}$. In particular, as shown in \cite{Fla22}*{Theorem 3.8}, we have that
$(q_n^s)^\perp=(\calH_n^{s+1})$,
where $\calH_n^{s+1}$ is the space of harmonic polynomials of degree $s+1$.
Therefore, as a direct consequence, we have that the $k$-th component of $(q_n^s)^\perp$ is equal to $0$ for every $k\leq s$, and hence, in particular, $\Cat_{q_n^s}^s$ has maximum rank. Thus, we immediately obtain the following proposition.
\begin{prop}[B.~Reznick]
\label{prop_lower_bound_catalecticant_powers_quadrics}
For every $n,s\in\bbN$,
\begin{equation*}
\rk (q_n^s)\geq\brk (q_n^s)\geq\rk\bigl(\Cat_{q_n^s}^s\bigr)=\binom{s+n-1}{s}.
\end{equation*}
\end{prop}

B.~Reznick provides in \cite{Rez92}*{formulas (8.35)--(8.36)} a decomposition of $q_n^2$ for $3\leq n\leq 7$, based on a family of integration quadrature formulas, which has been introduced by A.~H.~Stroud in \cite{Str67a}.
By verifying the formulas, we can prove that the same formula is also valid for $n\geq 9$. Hence, this gives an upper bound on the rank with the only exception of $n=8$, for which we can only get an upper bound of $45$. These decompositions can be found in \cite{Fla24}*{Theorems 5.9 and 5.15}.

\begin{teo}
\label{teo_minimal_decomposition_q_n^2}
Let $n\in\bbN$ be such that $n\geq 3$ and $n\neq 8$ and let $g\in\bbC$ such that $g^4=8-n$. Then, the form $q_n^2$ can be decomposed as
\begin{equation}
\label{rel_minimal_decomposition_q_n^2}
3a_5^4q_n^2=a_1\biggl(\sum_{j=1}^n x_{j}\biggr)^4+\sum_{k=1}^n\Biggl(a_2\biggl(\sum_{j=1}^nx_{j}\biggl)+a_3x_{k}\Biggr)^4+\sum_{j_1\neq j_2}\Biggl(a_4\biggl(\sum_{j=1}^nx_{j}\biggr)+a_5(x_{j_1}+x_{j_2})\Biggr)^4,
\end{equation}
where
\begin{gather*}
a_1=8(g^4-1)\bigl(g^2\pm 2\sqrt{2}\bigr)^4,\quad a_2=2g^2\pm 2\sqrt{2},\quad a_3=\mp 2\sqrt{2}g^4-8g^2,\\ a_4=2g,\qquad a_5=\bigl(\mp 2\sqrt{2}g^3-8g\bigr).
\end{gather*}
Moreover, the form $q_8^2$ can be decomposed as
\begin{equation}
\label{rel_minimal_decomposition_q_8^2}
q_8^2=-\frac{3}{512}\biggl(\sum_{j=1}^8x_{j}\biggr)^4-\frac{8}{9}\sum_{j=1}^8x_j^4+\frac{8}{9}\sum_{k=1}^8\Biggl(\frac{3}{16}{\biggl(\sum_{j=1}^8x_{j}\biggl)}+{x_{k}}\Biggr)^4+\frac{1}{3}\sum_{j_1\neq j_2}\Biggl(-\frac{3}{8}\biggl(\sum_{j=1}^8x_{j}\biggr)+(x_{j_1}+x_{j_2})\Biggr)^4.
\end{equation}
\end{teo}
The family of decompositions \eqref{rel_minimal_decomposition_q_n^2} is just one of the possible closed formulas that provide us a family of decompositions for the power of a quadratic form.  We now generalize this pattern by defining a new set of generators.

\begin{defn}
For every $s\in\bbN$, the \textit{set of $k$-partitions of $s$} is the set
\[
\calP_k(s)=\Set{(m_1,\dots,m_k)\in\bbN^k|\sum_{i=1}^km_i=s,\, m_1\geq\cdots\geq m_k>0}.
\]
The \textit{set of partitions of $s$} is the set
\[
\calP(s)=\bigcup_{j=1}^s\calP_j(s).
\]
\end{defn}
Given a natural number $n\in\bbN$, a problem of great relevance in both classical and recent literature is determining the number of partitions of $n$.
\begin{defn}
The \textit{partition function} is defined as the map
\[
\begin{tikzcd}[row sep=0pt,column sep=1pc]
 \p\colon \bbN\arrow{r} & \bbN\hphantom{.} \\
  {\hphantom{\p\colon{}}} s \arrow[mapsto]{r} & \abs*{\calP(s)}.
\end{tikzcd}
\]
For every $k\in\bbN$, the \textit{$k$-partition function} is defined as the map
\[
\begin{tikzcd}[row sep=0pt,column sep=1pc]
 \p_k\colon \bbN\arrow{r} & \bbN\hphantom{.} \\
  {\hphantom{\p_k\colon{}}} s \arrow[mapsto]{r} & \abs*{\calP_k(s)}.
\end{tikzcd}
\]
\end{defn}
We immediately observe that the equality
\[
\p(s)=\sum_{k=1}^s\p_k(s)
\]
holds for every $s\in\bbN$.
The partition function $\p(s)$ has always been of great importance in several branches of mathematics. Although no closed formulas are known to describe it, there are many asymptotic approximations and also many recurrence formulas providing the exact value of $\p(s)$ for every $s\in\bbN$. The first estimate of $\p(s)$ was provided in 1918 by G.~H.~Hardy and S.~Ramanujan in \cite{HR18}. In particular, they proved that, as $s\to +\infty$, it can be approximated as
\[
\p(s)\sim\frac{1}{4s\sqrt{3}}\rme^{\uppi\sqrt{\frac{2s}{3}}},
\]
which means that it grows as an exponential function of the square root of $s$. To get more details on the partition function $\p(s)$, we also refer to \cite{Rad37}, in which H.~Rademacher provides a convergent series for the partition function, improving the results contained in \cite{HR18}. For what concerns the $k$-partition function $\p_k(s)$, instead, there are some results regarding approximations and lower and upper bounds. 
In \cite{Mer14}*{p.~299}, M.~Merca provides an upper bound for the value $\p_k(s)$, from which we obtain the inequality
\begin{equation}
\label{rel_upper_bound_p_k(s)}
{\p}_{k}(s)\leq\frac{1}{2}\binom{s-1}{k-1}+\frac{1}{2}\leq \binom{s-1}{k-1}.
\end{equation}
Another tighter upper bound is given by A.~T.~Oruç in \cite{Oru16}*{Corollary 1}, corresponding to the formula
\[
\p_k(s)\leq\frac{5.44}{s-k}\rme^{\uppi\sqrt{\frac{2(s-k)}{3}}}.
\]

Denoting by $\mfS_n$ the permutation group of $n$ elements,
we introduce a family of polynomials which are invariant under the action of ${\mfS}_n$. For any partition $\bfm\in\calP_k(s)$, we define the polynomial
\begin{equation}
\label{rel_def_polynomials_p}
M_{\bfm}=\frac{1}{\abs{(\mfS_k)_{\bfm}}}\sum_{1\leq t_1<\cdots< t_k\leq n}\sum_{\sigma\in\mfS_k}x_{t_{\sigma(1)}}^{m_{1}}\cdots x_{t_{\sigma(k)}}^{m_{k}}=\frac{1}{\abs{(\mfS_n)_{\bfm}}}\sum_{\sigma\in\mfS_n}x_{\sigma(1)}^{m_1}\cdots x_{\sigma(n)}^{m_n},
\end{equation}
where $(\mfS_k)_{\bfm}$ denotes the stabilizer of the $k$-uple $\bfm$ under the action of $\mfS_k$. These polynomials are known as \textit{monomial symmetric polynomials}, which are defined in \cite{FH91}*{formula (A.2)}, and form a basis of the space of homogeneous symmetric polynomials. For further details about symmetric functions, we refer to \cite{FH91}*{Appendix A}.

In formula \eqref{rel_def_polynomials_p}, we consider every $k$-partition to be viewed as an $n$-tuple with the last entries equal to $0$.
Clearly, the polynomial $M_{\bfm}$ is invariant under the action of $\mfS_n$.
\begin{exam}
Let us consider the partition $(2,1,1)$ of $s=4$ in the case where $n=4$. Then, the polynomial $M_{(2,1,1)}\in\bbK[x_1,x_2,x_3,x_4]$, is given by
\begin{equation*}
M_{(2,1,1)}=\sum_{i=1}^4\sum_{\substack{1\leq j_1<j_2\leq 4\\ j_1,j_2\neq i}}x_i^2x_{j_1}x_{j_2}.
\end{equation*}
If instead we take the partition $(2,2,1,1)$ of $s=6$ in the case where $n=5$, we get
\[
M_{(2,2,1,1)}=\sum_{\substack{1\leq i_1<i_2\leq 5}}\sum_{\substack{1\leq j_1<j_2\leq 5\\ j_1,j_2\neq i_1,i_2}}x_{i_1}^2x_{i_2}^2x_{j_1}x_{j_2}.
\]
\end{exam}
The $s$-th power of the quadratic form $q_n$ can be developed as
\[
q_n^s=\sum_{i_1+\cdots+i_n=s}\binom{s}{i_1,\dots,i_n}x_1^{2i_1}\cdots x_{n}^{2i_n},
\]
for every $s\in\bbN$.
Gathering all multi-indices of the same permutation class, we can write
\begin{align}
\label{rel_formula_q_n^s_linear_comb_p}
q_n^s&\nonumber=\sum_{\bfm\in\calP(s)}\sum_{\sigma\in\mfS_n}\binom{s}{m_1,\dots,m_n}x_{\sigma(1)}^{2m_1}\cdots x_{\sigma(n)}^{2m_n}\\
&\nonumber=\sum_{k=1}^n\sum_{\bfm\in\calP_k(s)}\sum_{t_1\leq\cdots\leq t_k}\sum_{\sigma\in\mfS_k}\binom{s}{m_1,\dots,m_k}x_{t_{\sigma(1)}}^{2m_1}\cdots x_{t_{\sigma(k)}}^{2m_k}\\
&=\sum_{k=1}^n\sum_{\bfm\in\calP_k(s)}\binom{s}{m_1,\dots,m_k}M_{2\bfm}.
\end{align}
Now we need to remark a quite trivial fact.
\begin{lem}
\label{lem_odd_exponents}
For every $n$-uple $(a_1,\dots,a_n)\in\bbC^n$, the polynomial 
\[
\sum(a_1x_1\pm\cdots\pm a_nx_n)^{2k} \] 
contains only monomials with all exponents being even.
\end{lem}
\begin{proof}
The statement follows immediately from the fact that the polynomial
\[
\sum(a_1x_1\pm\cdots\pm a_nx_n)^{2k}
\]
represents an even function in each coordinate. That is, it remains unchanged by substituting the variable $x_i$ by the opposite $-x_i$ for every $i=1,\dots,n$.
\end{proof}
Considering equation \eqref{rel_decom_Rezn_q_n^2}, which corresponds to a family of decompositions of $q_n^2$,
we can show how to obtain it using \autoref{lem_odd_exponents}. It is sufficient to solve a linear system in two variables, where the unknowns are the coefficients of each monomial. These are the coefficients of the polynomials 
\[
M_{(2,2)}=\frac{1}{2!(n-2)!}\sum_{\sigma\in\mfS_n}x_{\sigma(1)}^2x_{\sigma(2)}^2=\sum_{i_1<i_2}x_{i_1}^2x_{i_2}^2,\qquad M_{(4)}=\frac{1}{(n-1)!}\sum_{\sigma\in\mfS_{n}}x_{\sigma(1)}^4=\sum_{i_1}x_{i_1}^4.
\]
Thus, by comparing the coefficients in both sides of the equation
\[
q_n^2=c_1\sum_{i_1<i_2}(x_{i_1}\pm x_{i_2})^4+c_2\sum_i^nx_i^4,
\]
i.e., writing
\[
2M_{(2,2)}=2\binom{4}{2}c_1M_{(2,2)},\qquad
M_{(4)}=2(n-1)c_1M_{(4)}+c_2M_{(4)},
\]
we obtain that the complete matrix of the system is a $2\times 3$ matrix equal to
\begin{equation}
\label{rel_matrix_q_n^2}
\begin{pNiceMatrix}[margin,vlines=3]
\Block[fill=gray!10,borders={top,left,right,bottom}]{1-1}{}12&0&2\\
2(n-1)&\Block[fill=gray!10,borders={top,left,right,bottom}]{1-1}{}1&1
\end{pNiceMatrix},
\end{equation}
which provides coefficients
$c_1=\frac 16$ and $c_2=\frac{4-n}{3}$.

Decomposition \eqref{rel_decomp_BHMT18_q_n^3} is a natural generalization of decomposition \eqref{rel_decom_Rezn_q_n^2} and we give a proof here, repeating the procedure for higher values of the exponent.
\begin{prop}[\cite{BHMT18}*{Section 4.5}]
\label{prop_BHMT18_q_n^3}
For every $n\in\bbN$, the form $q_n^3$ can be decomposed as
\begin{equation*}
60q_n^3=\sum_{i_1<i_2<i_3}(x_{i_1}\pm x_{i_2}\pm x_{i_3})^6+2(5-n)\sum_{i_1<i_2}(x_{i_1}\pm x_{i_2})^6+2(n^2-9n+38)\sum_{i_1}x_{i_1}^6,
\end{equation*}
which is a decomposition of size 
\[
4\binom{n}{3}+2\binom{n}{2}+n=\frac{2}{3}n^3-n^2+\frac{4}{3}n
\]
if $n\geq 3$ and $n\neq 5$ and of size $45$ if $n=5$.
\end{prop}
\begin{proof}
We can obtain the required equality by solving the linear system associated with the equation
\[
q_n^3=c_1\sum_{i_1<i_2<i_3}(x_{i_1}\pm x_{i_2}\pm x_{i_3})^6+c_2\sum_{i_1<i_2}(x_{i_1}\pm x_{i_2})^6+c_3\sum_{i=1}^nx_i^6.
\]
By \autoref{lem_odd_exponents} and the symmetry of each summand, we can suppose the development of the decomposition to be a linear combination of the polynomials  
\begin{equation*}
M_{(2,2,2)}=\sum_{i_1<i_2<i_3}x_{i_1}^2x_{i_2}^2x_{i_3}^2,\quad M_{(4,2)}=\sum_{i_1<i_2}\bigl(x_{i_1}^4x_{i_2}^2+x_{i_1}^2x_{i_2}^4\bigr),\quad
M_{(6)}=\sum_{i_1}x_{i_1}^6.
\end{equation*}
By comparing their coefficients, we get a linear system whose associated matrix is
\[
\begin{pNiceMatrix}[vlines=4,cell-space-limits=2pt]
2^2\binom{6}{2,2,2}&0&0&\binom{3}{1,1,1}\\
2^2(n-2)\binom{6}{4,2}&2\binom{6}{4,2}&0&\binom{3}{2,1}\\
2^2\binom{n-1}{2}&2(n-1)&1&1
\end{pNiceMatrix},
\]
which we can transform as
\begin{equation}
\label{rel_matrix_q_n^3}
\begin{pNiceMatrix}[margin,vlines=4,cell-space-limits=2pt]
\Block[fill=gray!10,borders={top,left,right,bottom}]{1-1}{}4&0&0&\frac{1}{15}\\
4(n-2)&\Block[fill=gray!10,borders={top,left,right,bottom}]{1-1}{}2&0&\frac{1}{5}\\
4\binom{n-1}{2}&2(n-1)&\Block[fill=gray!10,borders={top,left,right,bottom}]{1-1}{}1&1
\end{pNiceMatrix}
\end{equation}
By solving this linear system, we get the required decomposition.
\end{proof}
In this case, the size of decompositions given by formula \eqref{rel_decomp_BHMT18_q_n^3} grows as $n^3$ for $n\to +\infty$. Therefore, we can find a number $n_3\in\bbN$ such that 
\[
\rk(q^3_n)<\frac{1}{n}\binom{n+5}{6}
\] 
for every $n>n_3$.
In particular, we have 
\[
\frac{2}{3}n^3-n^2+\frac{4}{3}n<\frac{1}{n}\binom{n+5}{6}
\]
if and only if $n>11$. Thus we get $n_3=11$, which implies the following proposition.
\begin{prop}
\label{prop_subgeneric_rank_q_n^3}
Let $n\in\bbN$ be such that $n>11$. Then
\[
\rk(q_n^3)<\frac{1}{n}\binom{n+5}{6}.
\]
In particular, for $n>11$, $\rk(q_n^3)$ is subgeneric.
\end{prop}

\section{Further examples of closed formulas}
\label{sec_further examples}
The most natural idea to generalize this pattern for higher exponents would be to consider in the summation exactly the same kind of linear forms with additional variables, that is, linear forms of the type 
\[
\sum(x_{i_1}\pm\cdots\pm x_{i_s})^{2s}.
\]
Unfortunately, this condition is not sufficient to obtain powers of quadratic forms in general.
We can see this in the following example concerning the case of $s=4$. In this case, we have five polynomials, namely, $M_{(2,2,2,2)}$, $M_{(4,2,2)}$, $M_{(4,4)}$, $M_{(6,2)}$, and $M_{(8)}$. Assuming a hypothetical decomposition as
\begin{equation*}
q_n^4=c_1\sum_{i_1<i_2<i_3<i_4}(x_{i_1}\pm x_{i_2}\pm x_{i_3}\pm x_{i_4})^8+c_2\sum_{i_1<i_2<i_3}(x_{i_1}\pm x_{i_2}\pm x_{i_3})^8+c_3\sum_{i_1<i_2}(x_{i_1}\pm x_{i_2})^8+c_4\sum_{i_1}x_{i_1}^8,
\end{equation*}
we have four unknown coefficients $c_1,c_2,c_3,c_4\in\bbK$.
However, the matrix associated with the system obtained by comparing the coefficients of each monomial is
\[
\begin{pNiceMatrix}[vlines=5,cell-space-limits=2pt]
2^3\binom{8}{2,2,2,2}&0&0&0&\binom{4}{1,1,1,1}\\
2^3(n-3)\binom{8}{4,2,2}&2^2\binom{8}{4,2,2}&0&0&\binom{4}{2,1,1}\\
2^3\binom{n-2}{2}\binom{8}{4,4}&2^2(n-2)\binom{8}{4,4}&2\binom{8}{4,4}&0&\binom{4}{2,2}\\
2^3\binom{n-2}{2}\binom{8}{6,2}&2^2(n-2)\binom{8}{6,2}&2\binom{8}{6,2}&0&\binom{4}{3,1}\\
2^3\binom{n-1}{3}&2^2\binom{n-1}{2}&2(n-1)&1&1
\end{pNiceMatrix},
\]
whose rank is the same as the matrix
\[
\begin{pNiceMatrix}[vlines=5,cell-space-limits=2pt]
8&0&0&0&\frac{1}{105}\\
8(n-3)&4&0&0&\frac{1}{35}\\
4(n-2)(n-3)&4(n-2)&2&0&\frac{3}{35}\\
4(n-2)(n-3)&4(n-2)&2&0&\frac{1}{7}\\
4(n-1)(n-2)(n-3)&6(n-1)(n-2)&6(n-1)&3&3
\end{pNiceMatrix}.
\]
Comparing the third and the fourth rows of the matrix, we can see that the system has no solution. This shows that the condition of considering the same kind of linear forms with additional variables is not sufficient to obtain powers of quadratic forms in general for the case of $s=4$.

So, to obtain suitable decompositions for the case of exponent $4$, we need another set of points while maintaining the symmetry among all of the variables.
\begin{prop}
\label{prop_decomp_q_n^4}
For every $n\in\bbN$, the form $q_n^4$ can be decomposed as
\begin{align}
\label{rel_decompos_q_n^4}
\nonumber840q_n^4&=\sum_{i_1<i_2<i_3<i_4}(x_{i_1}\pm x_{i_2}\pm x_{i_3}\pm x_{i_4})^8+2(6-n)\sum_{i_1<i_2<i_3}(x_{i_1}\pm x_{i_2}\pm x_{i_3})^8\\[1ex]
&\nonumber\hphantom{{}={}}+\frac{2}{3}(3n^2-33n+76)\sum_{i_1<i_2}(x_{i_1}\pm x_{i_2})^8+\frac{2}{3}\sum_{i_1<i_2}\bigl[(2x_{i_1}\pm x_{i_2})^8+(x_{i_1}\pm 2x_{i_2})^8\bigr]\\[1ex]
&\hphantom{{}={}}-\frac{4}{3}\bigl(n^3-15n^2+317n-918\bigr)\sum_{i_1}x_{i_1}^8,
\end{align}
which is a decomposition of size
\[
8\binom{n}{4}+4\binom{n}{3}+6\binom{n}{2}+n=\frac{1}{3}(n^4-4n^3+14n^2-8n),
\]
for $n\geq 4$ and $n\neq 6$, and of size $216$ for $n=6$.
\end{prop}
\begin{proof}
As above, we consider the generic decomposition
\begin{align}
\label{rel_generic_linear_form_q_n^4}
q_n^4&\nonumber=c_1\sum_{i_1<i_2<i_3<i_4}(x_{i_1}\pm x_{i_2}\pm x_{i_3}\pm x_{i_4})^8+c_2\sum_{i_1<i_2<i_3}(x_{i_1}\pm x_{i_2}\pm x_{i_3})^8+c_3\sum_{i_1<i_2}(x_{i_1}\pm x_{i_2})^8\\[1ex]
&\hphantom{{}={}}+c_4\sum_{i_1<i_2}\bigl[(2x_{i_1}\pm x_{i_2})^8+(x_{i_1}\pm 2x_{i_2})^8\bigr]+c_5\sum_{i_1}x_{i_1}^8,
\end{align}
and the five different polynomials
\begin{gather*}
M_{(2,2,2,2)}=\sum_{i_1<i_2<i_3<i_4}x_{i_1}^2x_{i_2}^2x_{i_3}^2x_{i_4}^2,\quad M_{(4,2,2)}=\sum_{i_1<i_2<i_3}\bigl(x_{i_1}^4x_{i_2}^2x_{i_3}^2+x_{i_1}^2x_{i_2}^4x_{i_3}^2+x_{i_1}^2x_{i_2}^2x_{i_3}^4\bigr),\\
M_{(4,4)}=\sum_{i_1<i_2}x_{i_1}^4x_{i_2}^4,\quad
M_{(6,2)}=\sum_{i_1<i_2}\bigl(x_{i_1}^6x_{i_2}^2+x_{i_1}^2x_{i_2}^6\bigr),\quad M_{(8)}=\sum_{i_1}x_{i_1}^8.
\end{gather*}
By comparing the coefficients of these last elements with those of formula \eqref{rel_formula_q_n^s_linear_comb_p}, we obtain the matrix
\[
\begin{pNiceMatrix}[vlines=6,cell-space-limits=2pt]
2^3\binom{8}{2,2,2,2}&0&0&0&0&\binom{4}{1,1,1,1}\\
2^3(n-3)\binom{8}{4,2,2}&2^2\binom{8}{4,2,2}&0&0&0&\binom{4}{2,1,1}\\
2^3\binom{n-2}{2}\binom{8}{4,4}&2^2(n-2)\binom{8}{4,4}&2\binom{8}{4,4}&2(2^4+2^4)\binom{8}{4,4}&0&\binom{4}{2,2}\\
2^3\binom{n-2}{2}\binom{8}{6,2}&2^2(n-2)\binom{8}{6,2}&2\binom{8}{6,2}&2(2^6+2^2)\binom{8}{6,2}&0&\binom{4}{3,1}\\
2^3\binom{n-1}{3}&2^2\binom{n-1}{2}&2(n-1)&2(n-1)(2^8+1)&1&1
\end{pNiceMatrix},
\]
which, for the system, is equivalent to
\begin{equation}
\label{rel_matrix_q_n^4}
\begin{pNiceMatrix}[margin,vlines=6,cell-space-limits=2pt]
\Block[fill=gray!10,borders={top,left,right,bottom}]{1-1}{}8&0&0&0&0&\frac{1}{105}\\
8(n-3)&\Block[fill=gray!10,borders={top,left,right,bottom}]{1-1}{}4&0&0&0&\frac{1}{35}\\
4(n-2)(n-3)&4(n-2)&\Block[fill=gray!10,borders={top,left,right,bottom}]{2-2}{}2&64&0&\frac{3}{35}\\
4(n-2)(n-3)&4(n-2)&2&136&0&\frac{1}{7}\\
4(n-1)(n-2)(n-3)&6(n-1)(n-2)&6(n-1)&1542(n-1)&\Block[fill=gray!10,borders={top,left,right,bottom}]{1-1}{}3&3
\end{pNiceMatrix}.
\end{equation}
Since the highlighted blocks in matrix \eqref{rel_matrix_q_n^4} are invertible matrices, the linear system admits a unique solution, corresponding to the coefficients of decomposition \eqref{rel_decompos_q_n^4}.
\end{proof}
The crucial fact in determining decomposition \eqref{cor_log_limit} is that the entries of the blocks in matrix \eqref{rel_matrix_q_n^4} do not depend on $n$. This allows to obtain an explicit decomposition for an arbitrary number of variables. However, to generalize this pattern, it is necessary that all blocks have a non-zero determinant.

As observed in \autoref{prop_subgeneric_rank_q_n^3} for decomposition \eqref{rel_decomp_BHMT18_q_n^3}, one could verify that
\[
\frac{1}{3}(n^4-4n^3+14n^2-8n)<\frac{1}{n}\binom{n+7}{8}
\]
if and only if $n>10$. Consequently, we have the following proposition.
\begin{prop}
\label{prop_subgeneric_rank_q_n^4}
Let $n\in\bbN$ be such that $n>10$. Then
\[
\rk(q_n^4)<\frac{1}{n}\binom{n+7}{8}.
\]
In particular, for $n>10$, $\rk(q_n^4)$ is subgeneric.
\end{prop}

Before examining another example for the case of exponent $s=5$ we observe that decomposition \eqref{rel_decompos_q_n^4} is not optimal in general. By considering a different kind of linear forms for the fourth coefficient $c_4$, we can determine another decomposition of smaller size, involving roots of unity. Indeed, since $8$-th powers of linear forms are invariant under the action of $4$-roots of unity, we have, in particular,
\[
(x_{i_1}\pm \rmi x_{i_2})^8=(\rmi x_{i_1}\pm x_{i_2})^8.
\]
Thus, replacing the fourth summand of equation \eqref{rel_generic_linear_form_q_n^4} with this last linear form, we can repeat the same procedure and obtain the matrix
\begin{equation}
\label{rel_matrix_complex_q_n^4}
\begin{pNiceMatrix}[margin,vlines=6,cell-space-limits=2pt]
\Block[fill=gray!10,borders={top,left,right,bottom}]{1-1}{}8&0&0&0&0&\frac{1}{105}\\
8(n-3)&\Block[fill=gray!10,borders={top,left,right,bottom}]{1-1}{}4&0&0&0&\frac{1}{35}\\
4(n-2)(n-3)&4(n-2)&\Block[fill=gray!10,borders={top,left,right,bottom}]{2-2}{}2&2&0&\frac{3}{35}\\
4(n-2)(n-3)&4(n-2)&2&-2&0&\frac{1}{7}\\
4(n-1)(n-2)(n-3)&6(n-1)(n-2)&6(n-1)&6(n-1)&\Block[fill=gray!10,borders={top,left,right,bottom}]{1-1}{}3&3
\end{pNiceMatrix}.
\end{equation}
We therefore obtain decomposition of size
\[
8\binom{n}{4}+4\binom{n}{3}+4\binom{n}{2}+n=\frac{1}{3}n^4-\frac 43n^3+\frac{11}{3}n^2-\frac{5}{3}n,
\]
which is a slight improvement over decomposition \eqref{rel_decompos_q_n^4}.

To generalize this behavior, we provide another example for the case of exponent $s=5$. The proof is structured exactly like \autoref{prop_decomp_q_n^4}, but in this case, we have two blocks of size $2$ along the diagonal of the matrix.
\begin{prop}
\label{prop_decomp_q_n^5}
For every $n\in\bbN$, the form $q_n^5$ can be decomposed as 
\begin{align}
\label{rel_decompos_q_n^5}
15120q_n^5&\nonumber=\sum_{i_1<i_2<i_3<i_4<i_5}(x_{i_1}\pm x_{i_2}\pm x_{i_3}\pm x_{i_4}\pm x_{i_5})^{10}-2(n-7)\sum_{i_1<i_2<i_3<i_4}(x_{i_1}\pm x_{i_2}\pm x_{i_3}\pm x_{i_4})^{10}\\[1ex]
&\nonumber\hphantom{{}={}}+2(n^2-13n+36)\sum_{i_1<i_2<i_3}(x_{i_1}\pm x_{i_2}\pm x_{i_3})^{10}+\frac{2}{3}\sum_{i_1<i_2<i_3}\bigl[(2x_{i_1}\pm x_{i_2}\pm x_{i_3})^{10}\\[1ex]
&\nonumber\hphantom{{}={}}+(x_{i_1}\pm 2x_{i_2}\pm x_{i_3})^{10}+(x_{i_1}\pm x_{i_2}\pm 2x_{i_3})^{10}\bigr]+\frac{4}{3}(n^3-18n^2+90n-226)\sum_{i_1<i_2}(x_{i_1}\pm x_{i_2})^{10}\\[1ex]
&\nonumber\hphantom{{}={}}-\frac{4}{3}(n-4)\sum_{i_1<i_2}\bigl[(2x_{i_1}\pm x_{i_2})^{10}+(x_{i_1}\pm 2x_{i_2})^{10}\bigr]\\[1ex]
&\hphantom{{}={}}+\frac{2}{3}(n^4-22n^3+2195n^2-15086n+35592)\sum_{i_1}x_{i_1}^{10},
\end{align}
which is a decomposition of size 
\[
16\binom{n}{5}+8\binom{n}{4}+16\binom{n}{3}+6\binom{n}{2}+n=\frac{2}{15}n^5-n^4+\frac{16}{3}n^3-8n^2+\frac{68}{15}n
\]
for $n\geq 5$ and $n\neq 7$, and of size $1029$ for $n=7$.
\end{prop}
\begin{proof}
In this case, the polynomials involved for the decomposition are seven. These are
\begin{gather*}
M_{(2,2,2,2,2)}=\sum_{i_1<i_2<i_3<i_4<i_5}x_{i_1}^2x_{i_2}^2x_{i_3}^2x_{i_4}^2x_{i_5}^2,\\
M_{(4,2,2,2)}=\sum_{i_1<i_2<i_3<i_4}(x_{i_1}^4x_{i_2}^2x_{i_3}^2x_{i_4}^2+x_{i_1}^2x_{i_2}^4x_{i_3}^2x_{i_4}^2+x_{i_1}^2x_{i_2}^2x_{i_3}^4x_{i_4}^2+x_{i_1}^2x_{i_2}^2x_{i_3}^2x_{i_4}^4),\\
M_{(4,4,2)}=\sum_{i_1<i_2<i_3}(x_{i_1}^4x_{i_2}^4x_{i_3}^2+x_{i_1}^4x_{i_2}^2x_{i_3}^4+x_{i_1}^2x_{i_2}^4x_{i_3}^4),\quad
M_{(6,2,2)}=\sum_{i_1<i_2<i_3}(x_{i_1}^6x_{i_2}^2x_{i_3}^2+x_{i_1}^2x_{i_2}^6x_{i_3}^2+x_{i_1}^2x_{i_2}^2x_{i_3}^6),\\ M_{(6,4)}=\sum_{i_1<i_2}(x_{i_1}^6x_{i_2}^4+x_{i_1}^4x_{i_2}^6),\quad
M_{(8,2)}=\sum_{i_1<i_2}(x_{i_1}^8x_{i_2}^2+x_{i_1}^2x_{i_2}^8),\quad M_{(10)}=\sum_{i_1}x_{i_1}^{10}.
\end{gather*}
We can determine seven values $c_1,\dots,c_7$ such that
\begin{align*}
q_n^5&=c_1\sum_{i_1<i_2<i_3<i_4<i_5}(x_{i_1}\pm x_{i_2}\pm x_{i_3}\pm x_{i_4}\pm x_{i_5})^{10}+c_2\sum_{i_1<i_2<i_3<i_4}(x_{i_1}\pm x_{i_2}\pm x_{i_3}\pm x_{i_4})^{10}\\[1ex]
&\nonumber\hphantom{{}={}}+c_3\sum_{i_1<i_2<i_3}(x_{i_1}\pm x_{i_2}\pm x_{i_3})^{10}+c_4\sum_{i_1<i_2<i_3}\bigl((2x_{i_1}\pm x_{i_2}\pm x_{i_3})^{10}+(x_{i_1}\pm 2x_{i_2}\pm x_{i_3})^{10}\\[1ex]
&\nonumber\hphantom{{}={}}+(x_{i_1}\pm x_{i_2}\pm 2x_{i_3})^{10}\bigr)+c_5\sum_{i_1<i_2}(x_{i_1}\pm x_{i_2})^{10}+c_6\sum_{i_1<i_2}\bigl((2x_{i_1}\pm x_{i_2})^{10}+(x_{i_1}\pm 2x_{i_2})^{10}\bigr)+c_7\sum_{i_1}x_{i_1}^{10}.
\end{align*}
By comparing the coefficients, we obtain a linear system associated to the matrix
\begin{equation}
\label{rel_matrix_q_n^5}
\begin{pNiceMatrix}[margin,vlines=8]
\Block[fill=gray!10,borders={top,left,right,bottom}]{1-1}{}16&0&0&0&0&0&0&\binom{10}{2,2,2,2,2}^{-1}\binom{5}{1,1,1,1,1}\\
16(n-4)&\Block[fill=gray!10,borders={top,left,right,bottom}]{1-1}{}8&0&0&0&0&0&\binom{10}{4,2,2,2}^{-1}\binom{5}{2,1,1,1}\\
16\binom{n-3}{2}&8(n-3)&\Block[fill=gray!10,borders={top,left,right,bottom}]{2-2}{}4&144&0&0&0&\binom{10}{4,4,2}^{-1}\binom{5}{2,2,1}\\
16\binom{n-3}{2}&8(n-3)&4&288&0&0&0&\binom{10}{6,2,2}^{-1}\binom{5}{3,1,1}\\
16\binom{n-2}{3}&8\binom{n-2}{2}&4(n-2)&324(n-2)&\Block[fill=gray!10,borders={top,left,right,bottom}]{2-2}{}2&160&0&\binom{10}{6,4}^{-1}\binom{5}{3,2}\\
16\binom{n-2}{3}&8\binom{n-2}{2}&4(n-2)&1044(n-2)&2&520&0&\binom{10}{8,2}^{-1}\binom{5}{4,1}\\
16\binom{n-1}{4}&8\binom{n-1}{3}&4\binom{n-1}{2}&2052\binom{n-1}{2}&2(n-1)&2050(n-1)&\Block[fill=gray!10,borders={top,left,right,bottom}]{1-1}{}1&1
\end{pNiceMatrix}
\end{equation}
and solving it, we get the required decomposition.
\end{proof}
Again, we can compute the minimum number of variables that guarantees that the size of the decomposition is subgeneric.
\begin{prop}
\label{prop_subgeneric_rank_q_n^5}
Let $n\in\bbN$ be such that $n>8$. Then
\[
\rk(q_n^5)<\frac{1}{n}\binom{n+9}{10}.
\]
In particular, for $n>8$, $\rk(q_n^5)$ is subgeneric.
\end{prop}

\section{Asymptotic growth and upper bounds for Waring rank}
\label{sec_Asymptotic_growing}
To obtain a pattern similar to what we have seen in matrices \eqref{rel_matrix_q_n^4}, \eqref{rel_matrix_complex_q_n^4}, and \eqref{rel_matrix_q_n^5}, we need to construct suitable decompositions that provide us with a sufficiently large number of equations. This will ensure the existence of square matrices along the diagonal with non-zero determinants, i.e., the highlighted blocks we have seen above in matrices \eqref{rel_matrix_q_n^2}, \eqref{rel_matrix_q_n^3}, \eqref{rel_matrix_q_n^4}, \eqref{rel_matrix_complex_q_n^4}, and \eqref{rel_matrix_q_n^5}. To prove this, we start with a lemma for matrices of polynomials.
\begin{lem}
\label{lem_det_nonzero}
For every $n,s\in\bbN$, let $f_1,\dots,f_s\in\bbC[x_{1},\dots,x_{n}]$  be linearly independent polynomials and let 
\[
g\in\bbC[x_{ij}]_{i=1,\dots,s,\,j=1,\dots,n}
\]
be the polynomial defined as
\[
g(x_{11},\dots,x_{1n},\dots,x_{s1},\dots,x_{sn})=\det
\begin{pNiceMatrix}
f_1(x_{11},\dots,x_{1n})&\Cdots &f_1(x_{s1},\dots,x_{sn})\\
\Vdots&\Ddots&\Vdots\\
f_s(x_{11},\dots,x_{1n})&\Cdots &f_s(x_{s1},\dots,x_{sn})
\end{pNiceMatrix}.
\]
Then there exist $s$ points $\bfa_1,\dots,\bfa_s\in\bbC^n$ such that 
\[
g(\bfa_1,\dots,\bfa_s)\neq 0,
\]
namely, $g\not\equiv 0$.
\end{lem}
\begin{proof}
We prove the statement by induction on $s$. For $s=1$, the proof is trivial, since by linear independence we must have $f_1\not\equiv 0$. So, let us suppose that the statement is true for $s-1$ and consider the polynomial matrix
\[
A(x_{11},\dots,x_{1n},\dots,x_{s1},\dots,x_{sn})=
\begin{pNiceMatrix}
f_1(x_{11},\dots,x_{1n})&\Cdots &f_1(x_{s1},\dots,x_{sn})\\
\Vdots&\Ddots&\Vdots\\
f_s(x_{11},\dots,x_{1n})&\Cdots &f_s(x_{s1},\dots,x_{sn})
\end{pNiceMatrix}.
\]
Let us further consider the polynomials
\[
g_k\in\bbC[x_{ij}]_{i=1,\dots,s-1,\,j=1,\dots,n},
\]
defined as
\[
g_k=(-1)^{k+1}\det A_{kn}(x_{11},\dots,x_{1n},\dots,x_{(s-1)1},\dots,x_{(s-1)n})
\]
for every $k=1,\dots,s$, where $A_{kn}$ is the $(s-1)\times(s-1)$ matrix obtained by removing the $k$-th row and the $s$-th column from the matrix $A$. 
By the inductive hypothesis, we have that $g_k\not\equiv 0$, which means
\[
Z(g_k)\neq\bbC^{n(s-1)},
\]
and hence,
\[
D(g_k)=\bbC^{n(s-1)}\setminus Z(g_k)
\]
is a non-empty open set for every $k=1,\dots,s$. Therefore, we have
\[
\bigcap_{k=1}^{s}D(g_k)\neq\varnothing
\]
and we can select a point 
\[
\bfa=(\bfa_1,\cdots,\bfa_{s-1})\in\bbC^{n(s-1)}
\]
such that
$g_k(\bfa)\neq 0$
for every $k=1,\dots,s$. Now, considering the polynomial
\[
h(x_{s1},\dots,x_{sn})=g(\bfa_1,\dots,\bfa_{s-1},x_{s1},\dots,x_{sn})=\sum_{k=1}^sg_k(\bfa_1,\dots,\bfa_{s-1})f_k(x_{s1},\dots,x_{sn}),
\]
we have by the linear independence of the polynomials $f_1,\dots,f_s$, that
\[
h(x_{s1},\dots,x_{sn})\not\equiv 0.
\]
Therefore, we can select a point $\bfa_s\in\bbC^n$ such that
\[
h(\bfa_s)=g(\bfa_1,\dots,\bfa_s)\neq 0,
\]
proving the statement.
\end{proof}
The main blocks we want to consider have sizes equal to the different numbers of $k$-partitions of $s$, since they correspond to the number of monomials with exactly $k$ variables. In the following, we denote by $\bbN_0^k$ the set of $k$-tuple with strictly positive coordinates, for every $k\in\bbN$.

\begin{teo}
\label{teo_upper_bound_factorial}
For every $n,s\in\bbN$
\[
\rk(q_n^s)\leq\sum_{k=1}^s2^kk!\p_k(s)\binom{n}{k}.
\]
\end{teo}
\begin{proof}
As we have observed in the first cases above, we want to consider a sum of powers of linear form, while maintaining a strong symmetry between all the the variables.
For every $k=1,\dots,s$ and for every point 
$\bfa=(a_1,\dots,a_k)\in\bbC^k$, we introduce the polynomial $f_{k,\bfa}\in\bbK[x_1,\dots,x_n]$, defined as
\[
f_{k,\bfa}=\frac{1}{\abs{(\mfS_k)_{\bfa}}}\sum_{t_1<\cdots<t_k}\sum_{\sigma\in\mfS_k}(a_{\sigma(1)}x_{t_1}\pm\cdots\pm a_{\sigma(k)}x_{t_k})^{2s}.
\]
Here, we divide the sum by the quantity $\abs{(\mfS_k)_{\bfa}}$, as the action of the elements in $\mfS_k$ on $\bfa$ provides all the possible permutations of $\bfa$ without repetition. This structure was observed in formula \eqref{prop_decomp_q_n^4}, where we used the points 
\[
(1,1,1,1),\quad (1,1,1),\quad (1,1),\quad (2,1),\quad (1),
\]
and in formula \eqref{rel_decompos_q_n^5}, where we used the points
\[
(1,1,1,1,1),\quad (1,1,1,1),\quad (1,1,1),\quad (2,1,1),\quad (1,1),\quad (2,1),\quad (1).
\]
We develop the summations to separate monomials having the exponents in the same permutation class. We get
\begin{align*}
f_{k,\bfa}&=\frac{1}{\abs{(\mfS_k)_{\bfa})}}\sum_{t_1<\cdots<t_k}\sum_{\sigma\in\mfS_k}\sum_{\substack{\bfm\in\bbN^k\\[1pt] \abs{\bfm}=s}}2^{k-1}\binom{2s}{2m_1,\dots,2m_k}\prod_{i=1}^ka_{\sigma(i)}^{2m_i}x_{t_i}^{2m_i}\\
&=\frac{1}{\abs{(\mfS_k)_{\bfa}}}\sum_{t_1<\cdots<t_k}\sum_{\sigma\in\mfS_k}\sum_{\lambda=1}^k\sum_{\substack{\bfm\in\bbN_0^\lambda\\ \abs{\bfm}=s}}\sum_{1\leq s_1<\cdots<s_{\lambda}\leq k}2^{k-1}\binom{2s}{2m_1,\dots,2m_{\lambda}}\prod_{i=1}^\lambda a_{\sigma(s_i)}^{2m_i}x_{t_{s_i}}^{2m_i},
\end{align*}
where in the second equality, we have separated the monomials having a different number $\lambda$ of non-zero values appearing in the $k$-tuple $(m_1,\dots,m_k)$. We can also permute some summations, obtaining
\begin{align*}
f_{k,\bfa}&=\frac{1}{\abs{(\mfS_k)_{\bfa}}}\sum_{\lambda=1}^k\sum_{\substack{\bfm\in\bbN_0^\lambda\\ \abs{\bfm}=s}}\sum_{t_1<\cdots<t_k}\sum_{1\leq s_1<\cdots<s_{\lambda} \leq k}\sum_{\sigma\in\mfS_k}2^{k-1}\binom{2s}{2m_1,\dots,2m_{\lambda}}\prod_{i=1}^\lambda a_{\sigma(s_i)}^{2m_i}x_{t_{s_i}}^{2m_i}\\
&=\sum_{\lambda=1}^k\sum_{\substack{\bfm\in\bbN_0^\lambda\\ \abs{\bfm}=s}}\sum_{t_1<\cdots<t_k}\sum_{1\leq s_1<\cdots<s_{\lambda} \leq k}2^{k-1}\binom{2s}{2m_1,\dots,2m_{\lambda}}\biggl(\frac{1}{\abs{(\mfS_k)_{\bfa}}}\sum_{\sigma\in\mfS_k}\prod_{i=1}^\lambda a_{\sigma(s_i)}^{2m_i}\biggr)\prod_{j=1}^\lambda x_{t_{s_j}}^{2m_j}.
\end{align*}
The variables $x_{t_{s_1}},\cdots,x_{t_{s_{\lambda}}}$ appear among the variables $x_{t_1},\dots,x_{t_k}$ a number of times equal to the binomial coefficient
\[
\binom{n-\lambda}{k-\lambda}.
\]
Thus, we can remove the summation depending on $s_1,\dots,s_{\lambda}$ and write
\begin{align*}
f_{k,\bfa}&=\sum_{\lambda=1}^k\sum_{\substack{\bfm\in\bbN_0^\lambda\\ \abs{\bfm}=s}}\sum_{t_1<\cdots<t_{\lambda}}2^{k-1}\binom{n-\lambda}{k-\lambda}\binom{2s}{2m_1,\dots,2m_{\lambda}}\biggl(\frac{1}{\abs{(\mfS_k)_{\bfa}}}\sum_{\sigma\in\mfS_k}\prod_{i=1}^\lambda a_{\sigma(i)}^{2m_i}\biggr)\prod_{i=1}^\lambda x_{t_{i}}^{2m_i}\\
&=\sum_{\lambda=1}^k\sum_{\substack{\bfm\in\bbN_0^\lambda\\ \abs{\bfm}=s}}2^{k-1}\binom{n-\lambda}{k-\lambda}\binom{2s}{2m_1,\dots,2m_{\lambda}}\biggl(\frac{1}{\abs{(\mfS_k)_{\bfa}}}\sum_{\sigma\in\mfS_k}\prod_{i=1}^\lambda a_{\sigma(i)}^{2m_i}\biggr)\sum_{t_1<\cdots<t_{\lambda}}\prod_{i=1}^\lambda x_{t_{i}}^{2m_i}.
\end{align*}
Finally, we can gather the multi-indices representing the same partition of order $\lambda$ of the number $k$. To do this, it is sufficient to consider an element $\bfm=(m_1,\dots,m_{\lambda})$, such that 
$m_1\geq\cdots\geq m_{\lambda}>0$,
and its orbit under the action of the permutation group $\mfS_{\lambda}$.
We can decompose the polynomial $f_{k,\bfa}$ into more summands by distinguishing the order of each partition of $s$. Hence, we get
\begin{equation*}
f_{k,\bfa}
=\sum_{\lambda=1}^k\sum_{\bfm\in\calP_{\lambda}(s)}\sum_{\upgamma\in\mfS_{\lambda}}\frac{2^{k-1}}{\abs{(\mfS_{\lambda})_{\bfm}}}\binom{n-\lambda}{k-\lambda}\binom{2s}{2m_1,\dots,2m_{\lambda}}\biggl(\frac{1}{\abs{(\mfS_k)_{\bfa}}}\sum_{\sigma\in\mfS_k}\prod_{i=1}^\lambda a_{\sigma(i)}^{2m_{\gamma(i)}}\biggr)\sum_{t_1<\cdots<t_{\lambda}}\prod_{i=1}^\lambda x_{t_{i}}^{2m_{\gamma(i)}}.
\end{equation*}
Considering the summation
\[
\sum_{\sigma\in\mfS_k}\prod_{i=1}^\lambda a_{\sigma(i)}^{2m_{\gamma(i)}},
\]
since the permutation $\gamma$ just permutes the elements $1,\dots,\lambda$, it
varies among all the possible choices of the elements $a_1,\dots,a_k$ and, in particular, it is invariant under the action of $\mfS_{\lambda}$.
Therefore, we have
\[
\sum_{\sigma\in\mfS_k}\prod_{i=1}^\lambda a_{\sigma(i)}^{2m_{\gamma(i)}}=\sum_{\sigma\in\mfS_k}\prod_{i=1}^\lambda a_{\sigma\gamma^{-1}(i)}^{2m_{i}}=\sum_{\sigma\in\mfS_k}\prod_{i=1}^\lambda a_{\sigma(i)}^{2m_{i}}.
\]
and, by formula \eqref{rel_def_polynomials_p}, we obtain
\begin{align}
\label{rel_final_formula_f_k_a}
f_{k,\bfa}&\nonumber=\sum_{\lambda=1}^k\sum_{\bfm\in\calP_{\lambda}(s)}\frac{2^{k-1}}{\abs{(\mfS_{\lambda})_{\bfm}}}\binom{n-\lambda}{k-\lambda}\binom{2s}{2m_1,\dots,2m_{\lambda}}\biggl(\frac{1}{\abs{(\mfS_k)_{\bfa}}}\sum_{\sigma\in\mfS_k}\prod_{i=1}^\lambda a_{\sigma(i)}^{2m_i}\biggr)\sum_{\gamma\in\mfS_{\lambda}}\sum_{t_1<\cdots<t_{\lambda}}\prod_{i=1}^\lambda x_{t_{i}}^{2m_{\gamma(i)}}\\
&=\sum_{\lambda=1}^k\sum_{\bfm\in\calP_{\lambda}(s)}2^{k-1}\binom{n-\lambda}{k-\lambda}\binom{2s}{2m_1,\dots,2m_{\lambda}}\biggl(\frac{1}{\abs{(\mfS_k)_{\bfa}}}\sum_{\sigma\in\mfS_k}\prod_{i=1}^\lambda a_{\sigma(i)}^{2m_i}\biggr)M_{2\bfm}.
\end{align}
Thus, we have written the polynomial $f_{k,\bfa}$ as a linear combination of those given by the sum of monomials belonging to the same permutation class. Since these are the same ones appearing in the development of the form $q_n^s$ in formula \eqref{rel_formula_q_n^s_linear_comb_p}, we need to satisfy exactly $\p(s)$ conditions.

We establish these conditions by defining a polynomial $f_{k,\bfa_j}$ for every $j=1,\dots,\p_k(s)$ and repeating the same process for every $k=1,\dots,s$. This gives us a number of equations equal to
\[
\p(s)=\sum_{k=1}^s\p_k(s).
\]
In other words, we have to guarantee the existence, for every $k=1,\dots,s$, of a set of points 
\[
\calA_{k}=\{\bfa_{k,1},\dots,\bfa_{k,\p_k(s)}\},
\]
such that
\begin{equation}
\label{rel_q_n^s_linear combination_of_p_2m}
q_n^s=\sum_{k=1}^s\sum_{j=1}^{\p_k(s)}c_{k,j}f_{k,\bfa_{k,j}},
\end{equation}
for some coefficients $c_{k,j}\in\bbC$.
To find these points, we denote each coordinate of the points by considering $\bfa_{k,j}=(a_{k,j,1},\dots,a_{k,j,k})$
for every $j=1,\dots,\p_k(s)$.
Formula \eqref{rel_formula_q_n^s_linear_comb_p} tells us that
\[
q_n^s=\sum_{k=1}^s\sum_{\bfm\in\calP_k(s)}\binom{s}{m_1,\dots,m_k}M_{2\bfm}.
\]
Hence, for every $k=1,\dots,s$, we combine formulas \eqref{rel_formula_q_n^s_linear_comb_p}, \eqref{rel_final_formula_f_k_a} and \eqref{rel_q_n^s_linear combination_of_p_2m}, to obtain a linear equation for each partition $\bfm\in\calP_k(s)$. This equation is obtained by equalizing the coefficients of the polynomial $M_{2\bfm}$, leading to
\begin{equation}
\label{rel_equation_conditions}
\sum_{\lambda=k}^s\sum_{j_{\lambda}=1}^{\p_{\lambda}(s)}c_{\lambda,j_{\lambda}}\biggl(\frac{1}{\abs{(\mfS_k)_{\bfa_{\lambda,j_{\lambda}}}}}\sum_{\sigma\in\mfS_{\lambda}}\prod_{i=1}^\lambda a_{\lambda,j_{\lambda},\sigma(i)}^{2m_i}\biggr)=2^{1-k}\binom{n-\lambda}{k-\lambda}^{-1}\binom{2s}{2m_1,\dots,2m_{\lambda}}^{-1}\binom{s}{m_1,\dots,m_{\lambda}}.
\end{equation}
Formula \eqref{rel_q_n^s_linear combination_of_p_2m} provides a linear system of exactly $\p(s)$ equations of the type of formula \eqref{rel_equation_conditions}. The system has $\p(s)$ unknowns,  corresponding to the coefficients $c_{k,j_k}$ for $k=1,\dots,s$ and $j_k=1,\dots,\p_k(s)$. Denote each of the values which multiply coefficients $c_{k,j_k}$ as
\begin{equation}
\label{def_polynomials_h_i}
h_{k,j_k,\lambda}=h_{k,j_k}(\bfm_{\lambda})=\frac{1}{\abs{(\mfS_k)_{\bfa_{k,j}}}}\sum_{\sigma\in\mfS_{k}}\prod_{i=1}^k a_{k,j,\sigma(i)}^{2m_{\lambda,i}},
\end{equation}
where $\bfm_1,\dots\bfm_{\p_k(s)}\in\bbN_0^k$ are the $k$-partitions in $\calP_k(s)$. Now, the aim is to establish in which cases the linear system admits a solution. Thus we want to determine whenever the associated matrix is invertible.
This is given by
\begin{equation*}
\begin{pNiceMatrix}[margin,parallelize-diags=false]
  \Block[fill=gray!10,borders={top,left,right,bottom}]{1-1}{} h_{s,1,1}&0  &  \Cdots & & & & &  0\\
 *& \Block[fill=gray!10,borders={top,left,right,bottom}]{1-1}{} h_{s-1,1,1} &0  &\Cdots  & & & &0 \\
 & * & \ddots & & \\ 
  \Vdots & \Vdots & & \Block[fill=gray!10,borders={top,left,right,bottom}]{3-3}{} h_{k,1,1} &\Cdots & h_{k,\p_k(s),1} & &\Vdots\\
   &  &   &\Vdots  & \Ddots & \Vdots & &\\
   &  &  & h_{k,1,\p_k(s)}&\Cdots& h_{k,\p_k(s),\p_k(s)}\\
  & & & & & &\ddots &0\\
 * & * &\Cdots & & & & * &\Block[fill=gray!10,borders={top,left,right,bottom}]{1-1}{} h_{1,1,1}
\end{pNiceMatrix}
\end{equation*}
with several blocks on the diagonal of order $\p_k(s)$ for each $k=1,\dots,s$. In particular, the matrix is invertible if and only if each block has a non-zero determinant. For simplicity, we can assume that the coordinates of the points involved are pairwise distinct. This is possible because the set of points that do not annihilate each determinant is open.
Nevertheless, the number of points used to obtain such a decomposition can be much higher than the minimal ones. 
For each $k=1,\dots,s$, the elements $h_{k,j_k,\lambda}$ can be treated as polynomials with variables $a_{k,j_k,1},\dots,a_{k,j_k,k}$. Let us now look at the polynomials
\[
H_{k,\lambda}(x_1,\dots,x_k)=\frac{1}{\abs{(\mfS_k)_{\bfa_{k,j}}}}\sum_{\sigma\in\mfS_{k}}\prod_{i=1}^k x_{\sigma(i)}^{2m_{\lambda,i}},
\] 
with $\lambda=1,\dots,\rmp_k(s)$, representing the $k$-partitions in $\calP_k(s)$. These polynomials are obtained by replacing the variables $x_1,\dots,x_k$ by $a_{k,j_k,1},\dots,a_{k,j_k,k}$, that is, 
\[
h_{k,j_k,\lambda}=H_{k,\lambda}(a_{k,j_k,1},\dots,a_{k,j_k,k}),
\]
for every point $\bfa_{k,j_k}\in\bbC^k$, with $j_k=1,\dots,k$.
We want to prove that the polynomials $H_{k,\lambda}$ are linearly independent. So, let $c_1,\dots,c_{\rmp_k(s)}\in\bbC$ be coefficients such that
\[
c_1H_{k,1}+\cdots+c_{\rmp_k(s)}H_{k,\rmp_k(s)}=0.
\]
Every monomial appearing in this linear combination involves exactly $k$ variables. Furthermore, by the definition of $h_{k,j_k,\lambda}$, seen in formula \eqref{def_polynomials_h_i}, each of these monomials appears in only one of the polynomials $H_{k,\lambda}$. Indeed, every $H_{k,\lambda}$ involves only monomials of the same permutation class induced by the $k$-partition $\bfm_{\lambda}\in\calP_k(s)$. Therefore, since the permutation classes of monomials form a partition of the set of monomials with $k$ variables, the zero polynomial is obtained only if
\[
c_1=\cdots=c_{\rmp_k(s)}=0,
\]
which means that the polynomials $H_{k,\lambda}$ are linearly independent. 
Hence, we obtain the statement simply by applying \autoref{lem_det_nonzero}.
\end{proof}
The upper bound given by \autoref{teo_upper_bound_factorial} can be improved by considering only one point for the partitions of order $s$ and $s-1$, as these are unique and hence have only one condition to satisfy. 
\begin{proof}[Proof of \autoref{cor_upper_bound_quadrics}]
By choosing the points
\[
\bfa_{s,1}=\bfa_{s-1,1}=(1,\dots,1),
\]
in the proof of \autoref{teo_upper_bound_factorial}, we get the statement.
\end{proof}
By combining \autoref{prop_lower_bound_catalecticant_powers_quadrics} and \autoref{cor_upper_bound_quadrics}, we are able to prove obtain \autoref{cor_log_limit}.
\begin{proof}[Proof of \autoref{cor_log_limit}]
We have
\[
\binom{s+n-1}{s}=\frac{1}{s!}\prod_{j=1}^s(n-1+j)=\frac{1}{s!}\bigl(n^s+f(n)\bigr),
\]
where $f\in\bbC[n]$ is a polynomial of degree $s-1$. Thus, since
\[
\lim_{n\to +\infty}\frac{n^s}{n^s+f(n)}=1,
\]
we get
\[
\lim_{n\to +\infty}\log_n\Biggl(\binom{s+n-1}{s}\Biggr)=\lim_{n\to +\infty}\log_n\biggl(\frac{n^s}{s!}\biggr)=s.
\]
If we now look at the second member of inequality \eqref{rel_formula_upper_bound_corollary}, we get
\[
2^{s-1}\binom{n}{s}+2^{s-2}\binom{n}{s-1}+\sum_{k=1}^{s-2}2^{k-1}k!\p_k(s)\binom{n}{k}\leq 2^{s-1}\binom{n}{s}+2^{s-2}s!\binom{n}{s-1}\leq \frac{2^{s-1}}{s!}n^s+2^{s-2}sn^{s-1}.
\]
As above, we have
\[
\lim_{n\to +\infty}\frac{n^s}{n^s+\dfrac{s}{2}s!n^{s-1}}=1
\]
and hence
\[
\lim_{n\to +\infty}\log_n\Biggl(2^{s-1}\binom{n}{s}+2^{s-2}\binom{n}{s-1}+\sum_{k=1}^{s-2}2^{k-1}k!\p_k(s)\binom{n}{k}\Biggr)\leq\lim_{n\to +\infty}\log_n\biggl(\frac{2^{s-1}}{s!}n^s\biggr)=s.
\]
Therefore, by \autoref{prop_lower_bound_catalecticant_powers_quadrics} and \autoref{cor_upper_bound_quadrics}, we get the statement.
\end{proof}

\section{Subgeneric rank}
\label{sec:subgeneric_rank}
We have observed from \autoref{prop_subgeneric_rank_q_n^5} that the size of decomposition \eqref{rel_decompos_q_n^5} is subgeneric when $n>8$. Although it is not easy to determine the exact minimum value for which the higher numbers only give subgeneric rank, we can estimate that this value is quite low. Indeed, according to \autoref{cor_upper_bound_quadrics} and formula \eqref{rel_upper_bound_p_k(s)}, we obtain the inequality
\begin{align}
\label{rel_upper_bound_powers_quadrics}
\nonumber\rk(q_n^s)&\leq 2^{s-1}\binom{n}{s}+2^{s-2}\binom{n}{s-1}+\sum_{k=1}^{s-2}2^{k-1}k!\p_k(s)\binom{n}{k}\\[1ex]
&\leq 2^{s-1}\binom{n}{s}+2^{s-2}\binom{n}{s-1}+\sum_{k=1}^{s-2}\frac{2^{k-1}}{(k-1)!}\frac{(s-1)!n!}{(s-k)!(n-k)!}.
\end{align}
Meanwhile, the generic rank is equal to
\[
\frac{1}{n}\binom{2s+n-1}{2s}=\frac{1}{n}\biggl\lceil\frac{(2s+n-1)!}{(2s)!(n-1)!}\biggr\rceil.
\]
Hence, in the case of $n\to+\infty$, the value of the generic rank grows as $n^{2s-1}$. However, the upper bound in formula \eqref{rel_upper_bound_powers_quadrics} tells us that the rank of $q_n^s$ grows more slowly than $n^{s}$ as $n\to +\infty$.

Following the proof of \autoref{teo_upper_bound_factorial}, the goal is to determine some points $\bfa_{k,1},\dots,\bfa_{k,\p_k(s)}$ for every block matrix
\[
\begin{pNiceMatrix}
g_{k,1}(\bfa_{k,1})&\Cdots& g_{k,1}\bigl(\bfa_{k,\p_k(s)}\bigr)\\
 \Vdots &\Ddots&\Vdots\\
g_{k,\p_k(s)}(\bfa_{k,1})&\Cdots& g_{k,\p_k(s)}\bigl(\bfa_{k,\p_k(s)}\bigr)
\end{pNiceMatrix}
\]
such that its determinant is non-zero. Obviously, it would be convenient to choose as few distinct points as possible to get the smallest possible size of such decompositions.
Although this may be easy for small values of $k$, it becomes more difficult for larger sizes. 

Let us consider the upper bound of \autoref{teo_upper_bound_factorial}, given by
\begin{equation}
\label{rel_estimation_upper_bound}
\sum_{k=1}^s2^kk!\p_k(s)\binom{n}{k}.
\end{equation}
All the elements involved in the summation are quantities that have been extensively studied in the literature and for which many estimates and approximations have been provided.
A more precise estimate of the values involved in formula \eqref{rel_estimation_upper_bound} can be obtained using the classical and well-known Stirling's approximation to the factorial of an arbitrary natural number $n$ (see \cite{Sti30}, or \cite{Twe03} for a more recent version of the proof). This approximation shows that
\begin{equation}
\label{rel_Stirling_formula}
\sqrt{2\uppi}n^{n+\frac{1}{2}}\rme^{-n+\frac{1}{12n+1}}<n!<\sqrt{2\uppi}n^{n+\frac{1}{2}}\rme^{-n+\frac{1}{12n}}
\end{equation}
and hence, considering just the inverse of $n!$, we have
\begin{equation}
\label{rel_Stirling_formula_mutual}
\frac{1}{\sqrt{2\uppi}}n^{-n-\frac{1}{2}}\rme^{n-\frac{1}{12n}}<\frac{1}{n!}<\frac{1}{\sqrt{2\uppi}}n^{-n-\frac{1}{2}}\rme^{n-\frac{1}{12n+1}}.
\end{equation}
Now, considering the binomial coefficient
\[
\binom{n}{k}=\frac{n!}{k!(n-k)!},
\]
we can use relations \eqref{rel_Stirling_formula} and \eqref{rel_Stirling_formula_mutual} to provide the bounds
\begin{equation}
\label{rel_binomial_approximation_Stirling}
\frac{1}{\sqrt{2\uppi k}}\biggl(\frac{n}{n-k}\biggr)^{n+\frac{1}{2}}\biggl(\frac{n-k}{k}\biggr)^k\rme^{\frac{1}{12n+1}-\frac{n}{12k(n-k)}}<\binomial{n}{k}<\frac{1}{\sqrt{2\uppi k}}\biggl(\frac{n}{n-k}\biggr)^{n+\frac{1}{2}}\biggl(\frac{n-k}{k}\biggr)^k\rme^{\frac{1}{12n}-\frac{12n+1}{(12k+1)(12(n-k)+1)}}.
\end{equation}
Finally, we also provide an upper bound for the partition function $\p(s)$. D.~M.~Kane shows in \cite{Kan06}*{Remark 1} that for every $n\in\bbN$, we get 
\[
\frac{C^{-}_{1,1}}{n}\rme^{\uppi\sqrt{\frac{2n}{3}}}\leq\p(n)\leq\frac{C^{+}_{1,1}}{n}\rme^{\uppi\sqrt{\frac{2}{3}}},
\]
where $C_{1,1}^{-}$ and $C_{1,1}^{+}$ are some specified constant values. These inequalities are also used by A.~Y.~Oru\c{c} in \cite{Oru16}, where the author provides some upper bounds for the function $\p_k(n)$, to which we refer for further details about this. In particular, as explained for \cite{Oru16}*{formula (4)}, we can suppose $C_{1,1}^{+}=6$, getting
\begin{equation}
\label{rel_upper_bound_p(n)}
\p(n)\leq \frac{6}{n}\rme^{\uppi\sqrt{\frac{2n}{3}}}
\end{equation}
for every $n\in\bbN$.
\begin{proof}[Proof of \autoref{teo_subgeneric_rank}]
According to \autoref{teo_upper_bound_factorial}, we have the inequality
\begin{equation}
\rk(q_n^s)\leq\sum_{k=1}^s2^kk!\p_{k}(s)\binom{n}{k}\leq 2^ss!\binom{n}{s}\sum_{k=1}^s\p_k(s)=\frac{2^sn!}{(n-s)!}\p(s).
\end{equation}
By employing relations \eqref{rel_Stirling_formula}, \eqref{rel_Stirling_formula_mutual}, \eqref{rel_binomial_approximation_Stirling} and \eqref{rel_upper_bound_p(n)}, we can write
\begin{equation}
\label{rel_first_side_inequality_rank}
\frac{2^sn!}{(n-s)!}\p(s)\leq \frac{6\cdot 2^s}{s}\biggl(\frac{n}{n-s}\biggr)^{n+\frac{1}{2}}(n-s)^{s}\rme^{\uppi\sqrt{\frac{2s}{3}}+\frac{1}{12n}-\frac{1}{12(n-s)+1}-s}.
\end{equation}
Furthermore, by relation \eqref{rel_binomial_approximation_Stirling}, we have inequality
\begin{equation}
\label{rel_second_side_inequality_rank}
\frac{1}{n}\binom{2s+n-1}{2s}>\frac{1}{\sqrt{4\uppi s}}\biggl(\frac{2s+n-1}{n-1}\biggr)^{2s+n-\frac{1}{2}}\biggl(\frac{n-1}{2s}\biggr)^{2s}\rme^{-\frac{12n^2+72ns-23n+48s^2-70s+11}{24s(n-1)(12n+24s-11)}}
\end{equation}
Now, we want to see for which cases the second term of inequality \eqref{rel_first_side_inequality_rank} is lower than the second term of inequality \eqref{rel_second_side_inequality_rank}. So, we have
\[
\frac{6\cdot 2^s}{s}\biggl(\frac{n}{n-s}\biggr)^{n+\frac{1}{2}}(n-s)^{s}\rme^{\uppi\sqrt{\frac{2s}{3}}+\frac{1}{12n}-\frac{1}{12(n-s)+1}-s}<\frac{1}{\sqrt{4\uppi s}}\biggl(\frac{2s+n-1}{n-1}\biggr)^{2s+n-\frac{1}{2}}\biggl(\frac{n-1}{2s}\biggr)^{2s}\rme^{\frac{1}{12(2s+n-1)+1}-\frac{2s+n-1}{24s(n-1)}}
\] 
if and only if
\[
\frac{12\cdot 2^{s}\sqrt{\uppi}(n-1)}{\sqrt{s}(2s+n-1)}\biggl(\frac{n(n-1)}{(n-s)(2s+n-1)}\biggr)^{n+\frac{1}{2}}\biggl(\frac{2s\sqrt{n-s}}{2s+n-1}\biggr)^{2s}\rme^{\uppi\sqrt{\frac{2s}{3}}+\frac{1}{12n}+\frac{2s+n-1}{24s(n-2s)}-\frac{1}{12(2s+n-1)+1}-\frac{1}{12(n-s)+1}-s}<1
\]
and, making some calculations, we can see that 
\[
\frac{1}{2n}+\frac{2s+n-1}{24s(n-1)}-\frac{1}{12(2s+n-1)+1}<1
\]
for every $s>1$ and $n>(2s-1)^2$. Thus, we can also write
\[
\frac{12\cdot 2^{s}\sqrt{\uppi}(n-1)}{\sqrt{s}(2s+n-1)}\biggl(\frac{n(n-1)}{(n-s)(2s+n-1)}\biggr)^{n+\frac{1}{2}}\biggl(\frac{2s\sqrt{n-s}}{2s+n-1}\biggr)^{2s}\rme^{\uppi\sqrt{\frac{2s}{3}}-s+1}<1,
\]
which leads to
\[
\frac{(n-1)}{\sqrt{s}(2s+n-1)}2^{s}12\sqrt{\uppi}\rme^{\uppi\sqrt{\frac{2s}{3}}-s+1}\biggl(\frac{n(n-1)}{(n-s)(2s+n-1)}\biggr)^{n+\frac{1}{2}}\biggl(\frac{2s\sqrt{n-s}}{2s+n-1}\biggr)^{2s}<1.
\]
Now we need to analyze each factor and see when these are lower than $1$.  
First we observe that for $n>(2s-1)^2$ we have
\[
\frac{(n-1)}{\sqrt{s}(2s+n-1)}<1,\quad \frac{n(n-1)}{(n-s)(2s+n-1)}=\frac{n(n-1)}{n(n-1)+s(n-2s+1)}<1.
\]
Regarding the element,
\[
\frac{2s\sqrt{n-s}}{2s+n-1},
\]
we can compute that this is lower than $1$ if and only if
\[
4s^2(n-s)<4s^2+n^2+1+4sn-4s-2n,
\]
which simplifies to
\[
n^2-2(2s^2-2s+1)n+4s^3+4s^2-4s+1>0.
\]
This inequality holds, in particular, if
\[
n>(2s^2-2s+1)+2s\sqrt{s^2-3s+1},
\]
and hence also if
\[
n\geq(2s^2-2s+1)+2s\sqrt{s^2-2s+1}=(2s^2-2s+1)+2s(s-1)=(2s-1)^2.
\]
It remains to analyze the last factor $2^{s}12\sqrt{\uppi}\rme^{\uppi\sqrt{\frac{2s}{3}}-s+1}$, but it is not difficult to see that
\[
2^{s}12\sqrt{\uppi}\rme^{\uppi\sqrt{\frac{2s}{3}}-s+1}<1
\]
whenever $s\geq 95$. Although this estimate proves the statement for $s\geq 95$, it is possible to compute that, for $6\leq s\leq 94$, the upper bound provided by \autoref{teo_upper_bound_factorial} is lower than generic rank, i.e.,
\begin{equation}
\label{rel_inequality_generic_rank}
\sum_{k=1}^s2^kk!\p_k(s)\binom{n}{k}<\frac{1}{n}\binom{2s+n-1}{2s}
\end{equation}
 whenever $n> (2s-1)^2$.
First we observe that, if formula \eqref{rel_inequality_generic_rank} holds for some $s,n\in\bbN$, then it also holds for the same $s$ and $n+1$. Indeed, we have
\[
\sum_{k=1}^s2^kk!\p_k(s)\binom{n}{k}<\frac{1}{n}\binom{2s+n}{2s}=\frac{n+1}{2s+n}\Biggl(\frac{1}{n+1}\binom{2s+n-1}{2s}\Biggr),
\]
that is
\[
\frac{2s+n}{n+1}\sum_{k=1}^s2^kk!\p_k(s)\frac{n-k+1}{n+1}\binom{n+1}{k}<\frac{1}{n+1}\binom{2s+n-1}{2s}.
\]
In particular, for any $1\leq k\leq s$, we want to prove that
\[
\frac{(2s+n)(n-k+1)}{(n+1)^2}>\frac{(2s+n)(n-s+1)}{(n+1)^2}>1,
\]
where the last inequality is equivalent to
\[
n(s-1)>2s^2-2s+1.
\]
In particular, since $n>(2s-1)^2$, this is equivalent to
\[
n(s-1)>(4s^2-4s+1)(s-1)=4s^3-8s^2+5s-1,
\]
and we have
\[
4s^3-8s^2+5s-1>2s^2-2s+1
\]
if and only if
\[
4s^3-10s^2+7s-2>0.
\]
Since this last inequality holds for every $s\geq 2$, we simply have to prove that  formula \eqref{rel_inequality_generic_rank} holds for $6\leq s\leq 94$ and $n=(2s-1)^2+1$. Using the upper bound of formula \eqref{rel_upper_bound_p_k(s)}, it is easy to verify, by some computations, that formula \eqref{rel_inequality_generic_rank} holds for these values.
\end{proof}

\section*{Acknowledgements}
This paper is the natural continuation of the author's Ph.D.~thesis, which was pursued at Alma Mater Studiorum -- Università di Bologna, under the patient co-supervision of Alessandro Gimigliano and Giorgio Ottaviani, to whom sincere thanks are due. The author also extends special gratitude to Enrique Arrondo and Jaros\l{aw} Buczyński for their invaluable assistance and insightful suggestions during the author's Ph.D.~program, and to Vincenzo Galgano, Fulvio Gesmundo, Pierpaola Santarsiero, and Emanuele Ventura for their comments. The author is a member of the research group \textit{Gruppo Nazionale per le Strutture Algebriche Geometriche e Affini} (GNSAGA) of \textit{Istituto Nazionale di Alta Matematica} (INdAM) and is supported by the scientific project \textit{Multilinear Algebraic Geometry} of the program \textit{Progetti di ricerca di Rilevante Interesse Nazionale} (PRIN), Grant Assignment Decree No.~973 adopted on 06/30/2023 by the Italian Ministry of University and Research (MUR).

\bibliographystyle{amsalpha}
\bibliography{UpperBoundsRankPowersQuadraticFormsVersion2.bib}

\begin{bibdiv}
\begin{biblist}

\bib{AH95}{article}{
      author={Alexander, J.},
      author={Hirschowitz, A.},
       title={Polynomial interpolation in several variables},
        date={1995},
     journal={J. Algebraic Geom.},
      volume={4},
      number={2},
       pages={201\ndash 222},
}

\bib{BCC+18}{article}{
      author={Bernardi, A.},
      author={Carlini, E.},
      author={Catalisano, M.~V.},
      author={Gimigliano, A.},
      author={Oneto, A.},
       title={The hitchhiker guide to: secant varieties and tensor
  decomposition},
        date={2018},
     journal={Mathematics},
      volume={6},
      number={12},
       pages={Paper no.~314, 86 pp.},
}

\bib{BGI11}{article}{
      author={Bernardi, A.},
      author={Gimigliano, A.},
      author={Idà, M.},
       title={Computing symmetric rank for symmetric tensors},
        date={2011},
     journal={J. Symbolic Comput.},
      volume={46},
      number={1},
       pages={34\ndash 53},
}

\bib{BHMT18}{article}{
      author={Buczy\'{n}ski, J.},
      author={Han, K.},
      author={Mella, M.},
      author={Teitler, Z.},
       title={On the locus of points of high rank},
        date={2018},
     journal={Eur. J. Math.},
      volume={4},
      number={1},
       pages={113\ndash 136},
}

\bib{BO08}{article}{
      author={Brambilla, M.~C.},
      author={Ottaviani, G.},
       title={On the {A}lexander-{H}irschowitz theorem},
        date={2008},
     journal={J. Pure Appl. Algebra},
      volume={212},
      number={5},
       pages={1229\ndash 1251},
}

\bib{CGO14}{incollection}{
      author={Carlini, E.},
      author={Grieve, N.},
      author={Oeding, L.},
       title={Four lectures on secant varieties},
        date={2014},
   booktitle={in: \textit{Connections between algebra, combinatorics, and
  geometry} ({R}egina, {SK}, 2012), {Springer Proc. Math. Stat.}, {vol.~76},
  {Springer, New York}},
       pages={101\ndash 146},
}

\bib{Che99}{article}{
      author={Chevalier, P.},
       title={Optimal separation of independent narrow-band sources -- concept
  and performance},
        date={2011},
     journal={Signal Process.},
      volume={73{\normalfont{, special issue on blind separation and
  deconvolution}}},
       pages={27\ndash 48},
}

\bib{CS11}{article}{
      author={Comas, G.},
      author={Seiguer, M.},
       title={On the rank of a binary form},
        date={2011},
     journal={Found. Comput. Math.},
      volume={11},
      number={1},
       pages={65\ndash 78},
}

\bib{DC07}{article}{
      author={De~Lauthauwer, L.},
      author={Castaing, J.},
       title={Tensor-based techniques for the blind separation of ds-cdma
  signals},
        date={2007},
     journal={Signal Process.},
      volume={87},
       pages={322\ndash 336},
}

\bib{FH91}{book}{
      author={Fulton, W.},
      author={Harris, J.},
       title={Representation theory},
    subtitle={A first course},
      series={Graduate Texts in Mathematics},
   publisher={Springer-Verlag},
     address={New York},
        date={1991},
      volume={129, Reading in Mathematics},
}

\bib{Fla22}{article}{
      author={Flavi, C.},
       title={Border rank of powers of ternary quadratic forms},
        date={2023},
     journal={J. Algebra},
      volume={634},
       pages={599\ndash 625},
}

\bib{Fla24}{article}{
      author={Flavi, C.},
       title={Decomposition of powers of quadrics},
        date={2024},
      status={preprint, arXiv:2411.03161 [math.AG]},
}

\bib{GL19}{article}{
      author={Gesmundo, F.},
      author={Landsberg, J.~M.},
       title={Explicit polynomial sequences with maximal spaces of partial
  derivatives and a question of {K}. {M}ulmuley},
        date={2019},
     journal={Theory Comput.},
      volume={15},
       pages={Paper no. 3, 24 pp.},
}

\bib{GW98}{book}{
      author={Goodman, R.},
      author={Wallach, N.~R.},
       title={Representations and invariants of the classical groups},
      series={Encyclopedia of Mathematics and its Applications},
   publisher={Cambridge University Press},
     address={Cambridge},
        date={1998},
      volume={68},
}

\bib{HR18}{article}{
      author={Hardy, G.~H.},
       title={Asymptotic formulae in combinatory analysis},
        date={1918},
     journal={Proc. London Math. Soc.},
        note={Reprinted in \textit{Collected papers of Srinivasa Ramanujan
  (2)}, Vol.~17, Paper 34, 203-216, Amer.~Math.~Soc., New York, 2000,
  pp.~276--309},
}

\bib{IK99}{book}{
      author={Iarrobino, A.},
      author={Kanev, V.},
       title={Power sums, {G}orenstein algebras, and determinantal loci},
      series={Lecture Notes in Mathematics},
   publisher={Springer-Verlag},
     address={Berlin},
        date={1999},
      volume={1721},
        note={Appendix C by A.~Iarrobino and S.~L.~Kleiman},
}

\bib{Kan06}{article}{
      author={Kane, D.~M.},
       title={An elementary derivation of the asymptotics of partition
  functions},
        date={2006},
     journal={Ramanujan J.},
      volume={11},
      number={1},
       pages={49\ndash 66},
}

\bib{Lan12}{book}{
      author={Landsberg, J.~M.},
       title={Tensors: geometry and applications},
      series={Graduate Studies in Mathematics},
   publisher={American Mathematical Society},
     address={Providence, RI},
        date={2012},
      volume={128},
}

\bib{LO13}{article}{
      author={Landsberg, J.~M.},
      author={Ottaviani, G.},
       title={Equations for secant varieties of {V}eronese and other
  varieties},
        date={2013},
     journal={Ann. Mat. Pura Appl. (4)},
      volume={192},
      number={4},
       pages={569\ndash 606},
}

\bib{McC87}{book}{
      author={McCullagh, P.},
       title={Tensor methods in statistics},
      series={Monographs on Statistics and Applied Probability},
   publisher={Chapman \& Hall},
     address={London},
        date={1987},
}

\bib{Mer14}{article}{
      author={Merca, M.},
       title={New upper bounds for the number of partitions into a given number
  of parts},
        date={2014},
     journal={J. Number Theory},
      volume={142},
       pages={298\ndash 304},
}

\bib{Min84}{inproceedings}{
      author={Minkowski, H.},
       title={{G}rundlagen für eine {T}heorie der quadratischen {F}ormen mit
  ganzzahligen {K}oeffizienten},
        date={1884},
   booktitle={in: \textit{Mémoires présentés par divers savants a
  l’Académie des Sciences de l’institut national de France, Tome XXIX},
  {P}aper no.~2},
}

\bib{Oru16}{article}{
      author={Oru\c{c}, A.~Y.},
       title={On number of partitions of an integer into a fixed number of
  positive integers},
        date={2016},
     journal={J. Number Theory},
      volume={159},
       pages={355\ndash 369},
}

\bib{Rad37}{article}{
      author={Rademacher, H.},
       title={On the partition function p(n)},
        date={1937},
     journal={Proc. London Math. Soc. (2)},
      volume={43},
      number={4},
       pages={241\ndash 254},
}

\bib{Rez92}{article}{
      author={Reznick, B.},
       title={Sums of even powers of real linear forms},
        date={1992},
     journal={Mem. Amer. Math. Soc.},
      volume={96},
      number={463},
       pages={viii+155},
}

\bib{Ser73}{book}{
      author={Serre, J.-P.},
       title={A course in arithmetic},
      series={Graduate Texts in Mathematics},
   publisher={Springer-Verlag, New York-Heidelberg},
        date={1973},
      volume={No.~7},
        note={Translated from the French},
}

\bib{Sti30}{article}{
      author={Stirling, J.},
       title={Methodus differentialis: sive tractatus de summatione et
  interpolatione serierum infinitarum},
    language={Latin},
        date={1730},
}

\bib{Str67a}{article}{
      author={Stroud, A.~H.},
       title={Some fifth degree integration formulas for symmetric regions,
  {II}},
        date={1967},
     journal={Numer. Math.},
      volume={9},
       pages={460\ndash 468},
}

\bib{Syl51}{book}{
      author={Sylvester, J.~J.},
       title={An essay on canonical forms, supplement to a sketch of a memoir
  on elimination, transformation and canonical forms},
   publisher={George Bell \& Sons},
     address={London, Fleet Street},
        date={1851},
        note={Reprinted in: {\textit{The collected mathematical papers of James
  Joseph Sylvester}}, Vol.~1, paper no. 34, Chelsea Publishing Co., New York,
  1973, pp.~203--216, edited by H.~F.~Baker, reprint of the original edition
  published by Cambridge University Press, London, Fetter Lane, E.~C., 1904.},
}

\bib{Twe03}{book}{
      author={Tweddle, I.},
       title={James {S}tirling’s {M}ethodus {D}ifferentialis},
    subtitle={An annotated translation of {S}tirling's text},
   publisher={Springer-Verlag},
     address={London},
        date={2003},
}

\end{biblist}
\end{bibdiv}

\end{document}